\documentclass[smallextended,final,envcountsect]{svjour3} 
\smartqed 

\usepackage{amsmath,amssymb,mathrsfs}
\usepackage{color}
\usepackage{url,hyperref}
\usepackage[usenames,dvipsnames]{xcolor}
\usepackage{xfrac}
\usepackage{algorithm,algorithmic}
\usepackage{multicol} 
\usepackage{booktabs} 
\newtheorem{assumption}{Assumption}[section]

\newcommand{\R}{\mathbb{R}}
\newcommand{\Rext}{\R\cup\{+\infty\}}

\newcommand{\abs}[1]{\left\vert#1\right\vert}
\newcommand{\abss}[1]{\vert#1\vert}

\newcommand{\set}[1]{\left\{#1\right\}}
\newcommand{\sets}[1]{\{#1\}}
\newcommand{\norm}[1]{\left\Vert#1\right\Vert}
\newcommand{\norms}[1]{\Vert#1\Vert}
\newcommand{\tnorm}[1]{\vert\!\Vert#1\vert\!\Vert}
\newcommand{\tnorms}[1]{\vert\!\Vert#1\vert\!\Vert}

\newcommand{\Eproof}{\hfill $\square$}
\newcommand{\prox}{\mathrm{prox}}

\newcommand{\argmin}{\mathrm{arg}\!\displaystyle\min}

\newcommand{\dom}[1]{\mathrm{dom}(#1)}

\newcommand{\zero}[1]{{\boldsymbol{0}}}

\newcommand{\Xc}{\mathcal{X}}
\newcommand{\Yc}{\mathcal{Y}}
\newcommand{\Zc}{\mathcal{Z}}
\newcommand{\Sc}{\mathcal{S}}
\newcommand{\Sdp}{\mathbb{S}}

\newcommand{\Tc}{\mathcal{T}}

\newcommand{\Uc}{\mathcal{U}}
\newcommand{\Cc}{\mathcal{C}}
\newcommand{\Nc}{\mathcal{N}}

\newcommand{\Kc}{\mathcal{K}}

\newcommand{\Xopt}{\mathcal{X}^{\star}}
\newcommand{\Yopt}{\mathcal{Y}^{\star}}
\newcommand{\iprod}[1]{\left\langle #1\right\rangle}
\newcommand{\iprods}[1]{\langle #1\rangle}

\newcommand{\intx}[1]{\mathrm{int}\left(#1\right)}
\newcommand{\ri}[1]{\mathrm{ri}\left(#1\right)}

\newcommand{\BigO}[1]{\mathcal{O}\left(#1\right)}

\newcommand{\beforesec}{\vspace{-2.5ex}}
\newcommand{\aftersec}{\vspace{-2.5ex}}
\newcommand{\beforesubsec}{\vspace{-3ex}}
\newcommand{\aftersubsec}{\vspace{-2.5ex}}
\newcommand{\beforesubsubsec}{\vspace{-2ex}}
\newcommand{\aftersubsubsec}{\vspace{-1.25ex}}
\newcommand{\beforepara}{\vspace{-2.5ex}}

\begin{document}

\title{An Inexact Interior-Point Lagrangian Decomposition  Algorithm with Inexact Oracles}
\titlerunning{An Inexact Interior-Point Lagrangian Decomposition  Algorithm}


\author{Deyi Liu   \and  Quoc Tran-Dinh$^{\ast}$}
\authorrunning{Deyi Liu $\cdot$  Quoc Tran-Dinh}

\institute{Deyi Liu \and Quoc Tran-Dinh \at
             Department of Statistics and Operations Research, The University of North Carolina at Chapel Hill\\
              333 Hanes Hall, UNC-Chapel Hill, NC27599\\
              Emails: \{deyi@live.unc.edu, quoctd@email.unc.edu\}\\
              $^{\ast}$\textit{Corresponding author.}
}

\date{Received: date / Accepted: date}

\maketitle

\begin{abstract}
\vspace{-2ex}
We develop a new inexact interior-point Lagrangian decomposition method to solve a wide range class of constrained composite convex optimization problems.
Our method relies on four techniques: Lagrangian dual decomposition, self-concordant barrier smoothing, path-following, and proximal-Newton technique.
It also allows one to approximately compute the solution of the primal subproblems (called \textit{the slave problems}), which leads to inexact oracles (i.e., inexact gradients and Hessians) of the smoothed dual problem (called \textit{the master problem}).
The smoothed dual problem is nonsmooth, we propose to use an inexact proximal-Newton method to solve it.
By appropriately controlling the inexact computation at both levels: the slave and master problems, we still estimate a polynomial-time iteration-complexity of our algorithm as in standard short-step interior-point methods. 
We also provide a strategy to recover primal solutions and establish complexity to achieve an approximate primal solution.
We illustrate our method through two numerical examples on well-known models with both synthetic and real data and compare it with some existing state-of-the-art methods.
\end{abstract}

\keywords{Interior-point Lagrangian decomposition \and barrier smoothing \and inexact oracle \and proximal-Newton method \and constrained convex optimization}
\subclass{90C25 \and 90-08}


\beforesec
\section{Introduction}
\aftersec
The Lagrangian dual decomposition framework is a classical technique to handle constrained convex optimization problems with separable structures such as conic, multi-stage stochastic, network, and distributed optimization problems \cite{Bertsekas1989b,Birge1985,Connejo2006,Dantzig1963,Palomar2006}. 
This approach has been incorporated with interior-point methods to obtain a dual decomposition interior-point framework in early 1990s \cite{Kojima1993}.
Since then, many researchers have regularly applied this approach to different problems.
For example, \cite{Fukuda2002} exploited this idea to develop a dual decomposition algorithm for semidefinite programming, and \cite{Zhao2005} considered this method for general convex and multi-stage stochastic programming.
The authors in \cite{Necoara2009} further investigated the method from \cite{Zhao2005} to solve a more general class of problems and obtained more intensive and rigorous complexity guarantees.
The work \cite{TranDinh2012c} studied this framework under the effect of inexact oracle computed by inexactly solving  the primal subproblems up to a given accuracy.
Other related theoretical results include \cite{bitlislioglu2017interior,fukuda2000interior,gros2014newton,halldorsson2003interior,Pakazad2017,Shida2008,Yamashita2011}.
In particular, \cite{Pakazad2017} solved loosely coupled problems  using message passing, and \cite{bitlislioglu2017interior} applied it to multi-agent optimization problems.
However, none of these works has investigated general constrained composite convex optimization settings involving linear operators and allows both inexactness in the slave problems and master problem altogether. 
In addition, existing methods do not handle directly nonsmooth objectives but often introduce auxiliary variables to reformulate the underlying problem into a smooth problem  which may significantly increase problem size and loose their theoretical guarantee.

\beforepara
\paragraph{\textbf{Motivation and goals:}}
Although the Lagrangian decomposition method is classical, it is very useful to handle large-scale constrained convex problems with separable structure by means of parallel and distributed computational architectures. 
In this paper, we conduct an intensive study on the interior-point Lagrangian decomposition (IPLD) framework considered in many existing works, especially \cite{Necoara2009,TranDinh2012c,Zhao2005}, from the following aspects.
\begin{itemize}
\vspace{-1ex}
\item[(a)] Firstly, we consider a more general problem class than \cite{Kojima1993,Necoara2009,TranDinh2012c,Zhao2005} by handling directly a nonsmooth composite convex function with a linear operator (see \eqref{eq:constr_cvx}) instead of couple linear equality constraints as in existing methods by means of proximal Newton-type methods (see Subsection~\ref{subsec:inexact_newton_dual}). 

\item[(b)] Secondly, our method works with inexact oracles of the dual problem arising from inexactly solving the primal subproblems (the slave problems).
We explicitly describe the range of accuracies to flexibly control the tolerance of the subproblems (see Subsection~\ref{subsec:inexact_oracles}).

\item[(c)] Thirdly, we also exploit inexact proximal-Newton method to handle general nonsmooth terms of the dual problems.

\item[(d)] Fourthly, we provide a thorough analysis for both the primal and dual problems and derive concrete iteration-complexity bounds for our method.

\item[(e)] Finally, we incorporate our approach with a recent concept called ``generalized self-concordance'' developed in \cite{SunTran2017gsc} to handle new applications.

\vspace{-1ex}
\end{itemize}
We are interested in the class of constrained composite convex problems where $g$ is smooth and satisfies some additional properties so that existing methods often do not have a theoretical convergence guarantee.
For instance, the objective function does not have Lipschitz gradient or is not ``tractably proximal''.
We also consider a generic convex set where the projection onto it may not be tractable to compute such as general polyhedra.
Under such assumptions, our problem setting covers a wide range class of applications ranging from optimal control, operations research, and networks to machine learning, statistics, and signal processing \cite{BenTal2001,Boyd2004}.
It also covers standard conic programming such as linear programming, second-order cone programming, and semidefinite programming.

\beforepara
\paragraph{\textbf{Our contribution:}}
We exploit the approach from \cite{Kojima1993,Necoara2009,TranDinh2012c,Zhao2005} to develop a new algorithm for solving a class of constrained convex optimization problems.
The main idea is to smooth the dual problem using a self-concordant barrier function \cite{Nesterov1994} associated with the constraint set, and apply a path-following scheme to solve the smoothed dual problem. 
While  \cite{Kojima1993,Necoara2009,Zhao2005} exactly follow this main stream,  \cite{TranDinh2012c} proposed another path-following scheme and analyzed its convergence under inexact computation.
It also provides a strategy to recover an approximate primal solution from its approximate dual  solution.
Compared to \cite{TranDinh2012c}, this work studies a much more general problem class than \cite{TranDinh2012c}.
In addition, it is different from existing works, including \cite{TranDinh2012c}, in several aspects as previously mentioned.
To this end, we can summarize our contribution as follows:
\begin{itemize}
\vspace{-1ex}

\item[(a)]
We exploit the approach in \cite{Kojima1993,Necoara2009,TranDinh2012c,Zhao2005} and combine it with recent new mathematical tools in \cite{Nesterov2004c,SunTran2017gsc} to develop a new algorithm.
The new mathematical tools allow us to cover much broader class of models than \cite{Kojima1993,Necoara2009,TranDinh2012c,Zhao2005}, and to analyze polynomial-time iteration-complexity.
In addition, we handle a more general class of problems than  \cite{Kojima1993,Necoara2009,TranDinh2012c,Zhao2005} by allowing general composite convex objectives involving linear operators (see \eqref{eq:constr_cvx}).

\item[(b)]
We propose a new inexact interior-point Lagrangian decomposition algorithm to solve this class of problems.
Our algorithm can deal with inexact oracles of the dual problems arising from approximating the primal subproblem solutions.
It also uses an inexact proximal-Newton scheme to approximate the search direction in the dual problem.
We characterize explicitly the choice of all related parameters and accuracies based on our analysis.

\item[(c)]
We establish a polynomial-time iteration-complexity estimate of our method to find an approximate optimal solution. 
Our algorithm can be viewed as a short-step interior-point methods for general convex problems involving Nesterov and Nemirovskii's self-concordance structures.
Our complexity bound is the same as in standard interior-point methods (up to a constant factor), while it is able to directly handle nonsmooth objective by means of proximal operator.
\vspace{-1.5ex}
\end{itemize}

In addition to the above main contribution, let us highlight some technical contribution of our methods.
Firstly, unlike other methods involving inexact oracles in the literature \cite{Devolder2010}, our inexact oracle is rendered from inexact solution of the subproblem.
The accuracy level can be adaptively chosen instead fixing as in existing methods to flexibly trade-off the computation cost by choosing rough accuracy  at the early iterations and decrease it in the last iterations.
Secondly, solving the primal subproblem (slave problem) is reduced to solve a nonlinear equation instead of a general convex problem as in some existing decomposition methods.
As a result, we can characterize an implementable criterion to control the inexactness of the primal subproblems by using Newton-type schemes.
Thirdly, instead of using unspecified parameters such as the radius of quadratic convergence region and contraction factor, we compute these parameters explicitly using the theory of self-concordant barriers as often seen in interior-point methods \cite{Nesterov2004,Nesterov1994}.
Finally, combining inexact oracle and inexact methods make our algorithm practical since this computation is unavoidable in iterative methods, especially, in decomposition approaches when handling complex models.

\beforepara
\paragraph{\textbf{Paper organization:}}
The rest of this paper is organized as follows.
Section~\ref{sec:prob_statement} states the problem of interest, basic assumptions, and its dual form.
Section~\ref{sec:gsc} recalls some preliminary results on (generalized) self-concordance and self-concordant barriers \cite{Nesterov1994}.
Section~\ref{sec:barrier_smoothing} focuses on barrier smoothing techniques and inexact oracles.
Section~\ref{sec:dual_decomp_alg} presents our main algorithm and its complexity analysis as well as convergence guarantees.
Section~\ref{sec:numerical_experiment} provides two numerical examples to verify the theoretical results.
For the sake of presentation, we move all the technical proofs to the appendix.

\beforesec
\section{Problem statement, basic assumptions, and dual formulation}\label{sec:prob_statement}
\aftersec
\paragraph{\textbf{Notation and terminologies:}}
We work with finite dimensional vector space $\R^p$ or $\R^n$ endowed with standard inner product $x^{\top}y$ or $\iprods{x, y}$ and Euclidean norm $\norm{x}_2 := \sqrt{x^{\top}x}$.
We denote by $\Sdp^p_{+}$ (resp., $\Sdp^p_{++})$ the set of symmetric positive semidefinite matrices (resp., symmetric positive definite matrices).
Given $H\in\Sdp^p_{++}$, we define a weighted norm $\norms{u}_H := \left(u^{\top}Hu\right)^{1/2}$ and its dual norm $\norms{v}_{H}^{\ast} := \left(v^{\top}H^{-1}v\right)^{1/2}$ for any vectors $u$ and $v$ in $\R^p$.
For $X, Y\in\Sdp^p_{+}$, $X\preceq Y$ means that $Y-X\in\Sdp^p_{+}$ and $X\succeq Y$ stands for $X-Y\in\Sdp^p_{+}$.

Given a three-time differentiable and strictly convex function $f$, we define the following local norms for any $u$ and $v$ in $\R^p$:
\begin{equation}\label{eq:local_norms}
\norms{u}_x := \left(u^{\top}\nabla^2{f}(x)u\right)^{1/2}, ~~~\text{and}~~~\norms{v}_x^{\ast} := \left(v^{\top}\nabla^2{f}(x)^{-1}v\right)^{1/2}.
\end{equation}
They also satisfy the Cauchy-Schwarz inequality, i.e. $u^{\top}v \leq \norms{u}_x\norms{v}_x^{\ast}$.
We say that $f$ is $\mu_f$-strongly convex if $f(\cdot) - (\mu_f/2)\norms{\cdot}^2$ remains convex.
We also often use the following two convex functions: $\omega(\tau) := \tau - \ln(1+\tau)$ for $\tau \geq 0$, and $\omega_{\ast}(\tau) := -\tau - \ln(1 - \tau)$ for $\tau \in [0, 1)$.
These functions are smooth and strictly convex.
We also use $\BigO{\cdot}$ to denote big-O complexity notion.

\beforesubsec
\subsection{\bf The primal problem and basic assumptions}
\aftersubsec
Consider the following constrained composite convex optimization problem:
\begin{equation}\label{eq:constr_cvx}
P^{\star} := \min_{x\in\Kc}\Big\{ P(x) := g(x)  + \phi(Ax) \Big\},
\end{equation}
where $g : \R^p\to\R$ is a smooth and convex function, $\phi : \R^n\to\Rext$ is a proper, closed,  and convex function, $A\in\R^{n\times p}$, and $\Kc$ is a nonempty, closed, and convex set in $\R^p$. 
As a special case of \eqref{eq:constr_cvx}, if we choose $\phi := \delta_{\Cc}$, the indicator of a nonempty, closed, and convex set $\Cc$ in $\R^n$, then \eqref{eq:constr_cvx} reduces to the following general constrained convex problem:
\begin{equation}\label{eq:constr_cvx2}
g^{\star} := \min_{x\in\Kc}\Big\{ g(x) ~~~\mathrm{s.t.}~~Ax \in \Cc \Big\}.
\end{equation}
Without loss of generality, we can also assume that $g$ and $\Kc$ possess a separable structure as follows:
\vspace{-0.75ex}
\begin{equation}\label{eq:separable_structure}
g(x) := \sum_{i=1}^Ng_i(x_i)~~~\text{and}~~~\Kc := \Kc_1\times \cdots \times \Kc_N,
\vspace{-0.75ex}
\end{equation}
for $N\geq 1$, where $x_i\in\R^{p_i}$, $\Kc_i\subseteq\R^{p_i}$, and $\sum_{i=1}^Np_i = p$ for $i=1,\cdots, N$.

Note that the separable structure \eqref{eq:separable_structure} frequently appears in graph and network optimization.
It is also a natural structure in conic programming such as linear programming and monotropic programming \cite{Rockafellar1985}.
Another example is convex empirical minimization models in statistical learning, which can also be reformulated into \eqref{eq:constr_cvx} by duplicating variables.

\beforepara
\paragraph{\textbf{Basic assumptions:}}
Our approach relies on the following assumptions:
\begin{assumption}\label{as:A0}
The optimal solution set $\Xopt$ of \eqref{eq:constr_cvx} is nonempty, and hence the optimal value $P^{\star}$ is finite. 
The following Slater condition holds:
\vspace{-0.75ex}
\begin{equation}\label{eq:Slater_cond}
0 \in  \ri{\dom{\phi} - A(\dom{g}\cap\Kc)},
\vspace{-0.75ex}
\end{equation}
where $\ri{\Zc}$ is the relative interior of $\Zc$, and $\dom{\cdot}$ is the domain of $(\cdot)$.
\end{assumption}

\begin{assumption}\label{as:A1}
The function $g$ is standard self-concordant as in Definition~\ref{de:gsc_def}. 
$\Kc$ is endowed with a $\nu_f$-self-concordant barrier $f$ as in Definition~\ref{de:scb_function}
and $A$ is full-row rank.
\end{assumption} 

Note that Assumption~\ref{as:A0} is standard and required in any primal-dual optimization method to guarantee strong duality.
Assumption~\ref{as:A1} is also not restrictive. 
First, the self-concordance of $g$ can be relaxed to a broader class called generalized self-concordant function as shown in Proposition~\ref{pro:conjugate} with additional structures.
Next, the full-row rankness of $A$ can always be obtained by eliminating redundant rows.
Finally, the self-concordant barrier of $\Kc$ is always guaranteed under mild condition as discussed in \cite{Nesterov1994}.

Throughout this paper, we assume that both Assumptions~\ref{as:A0} and \ref{as:A1} hold without recalling them in the sequel.

\beforesubsec
\subsection{\bf Dual problem and optimality condition}\label{subsec:pd_theory}
\aftersubsec
The dual problem associated with \eqref{eq:constr_cvx} can be written as
\vspace{-0.75ex}
\begin{equation}\label{eq:dual_prob}
D^{\star} := \min_{y\in\R^n}\bigg\{ D(y) := \underbrace{\max_{x\in\Kc}\bigg\{ \iprods{Ax, y} - g(x) \bigg\}}_{d(y)} ~+~ \phi^{\ast}(-y) \bigg\},
\vspace{-0.75ex}
\end{equation}
where $\phi^{\ast}(\cdot) := \sup_{u}\set{\iprods{\cdot, u} - \phi(u)}$ is the Fenchel conjugate of $\phi$.
Under the separable structure \eqref{eq:separable_structure}, we can decompose the dual function $d$ into $N$ functions $d_i$ on smaller spaces $\R^{p_i}$.
That is 
\vspace{-0.75ex}
\begin{equation*}
d(y) := \sum_{i=1}^Nd_i(y) ~~~~\text{with}~~~d_i(y) := \max_{x_i\in\Kc_i}\Big\{ \iprods{A_ix_i, y} - g_i(x_i) \Big\}.
\vspace{-0.75ex}
\end{equation*}
This computation can be carried out in parallel.
Moreover, under Assumption~\ref{as:A0}, the dual optimal solution set $\Yopt$ of \eqref{eq:dual_prob} is nonempty, and the strong duality holds, i.e. $P^{\star} + D^{\star} = 0$.
The optimality condition of the primal problem \eqref{eq:constr_cvx} can be written as
\vspace{-0.75ex}
\begin{equation}\label{eq:opt_cond}
\left\{\begin{array}{lll}
0 &\in \nabla{g}(x^{\star}) - A^{\top}y^{\star} + \Nc_{\Kc}(x^{\star}), ~~~&\text{(primal optimality)} \vspace{1ex}\\
0 & \in  y^{\star}  + \partial{\phi}(Ax^{\star}) &\text{(dual optimality)}\vspace{1ex}\\
x^{\star} &\in \Kc, &\text{(primal feasibility)}.
\end{array}\right.
\vspace{-0.75ex}
\end{equation}
Under Assumption~\ref{as:A0}, \eqref{eq:opt_cond} is the necessary and sufficient condition for $x^{\star} \in \Xc^{\star}$ to be a primal optimal solution of \eqref{eq:constr_cvx}, and $y^{\star}\in\Yc^{\star}$ to be a dual optimal solution of \eqref{eq:dual_prob}.
Note that  $0 \in y^{\star} + \partial{\phi}(Ax^{\star})$ can be written as
\vspace{-0.75ex}
\begin{equation}\label{eq:dual_optimality_cond}
0 \in Ax^{\star} - \partial{\phi^{\ast}}(-y^{\star}) \equiv \nabla{d}(y^{\star})  - \partial{\phi^{\ast}}(-y^{\star}).
\vspace{-0.75ex}
\end{equation}
This is exactly the optimality condition of the dual problem \eqref{eq:dual_prob}.
Our goal is to approximate a primal-dual solution of \eqref{eq:constr_cvx} and \eqref{eq:dual_prob} in the sense of Definition~\ref{de:primal_dual_approx_sols}.

\beforesec
\section{Generalized self-concordance and self-concordant barriers}\label{sec:gsc}
\aftersec
Let us review the theory of generalized self-concordant functions \cite{SunTran2017gsc} and self-concordant barriers \cite{Nesterov2004,Nesterov1994}, which will be used in the sequel.

\beforepara
\paragraph{\textbf{Generalized self-concordance and standard self-concordance:}}
Assume that $f : \dom{f}\subseteq\R^p\to\R$ is a three-time continuously differentiable convex function, i.e. $f\in\mathbb{C}^3(\dom{f})$, we use $\nabla^3{f}(x)[u]$ to denote the third order derivative of $f$ at $x\in\dom{f}$ along a direction $u\in\R^p$.
We recall the following definition \cite{SunTran2017gsc}. 

\begin{definition}[\cite{SunTran2017gsc}]\label{de:gsc_def}
A $\mathbb{C}^3$-convex function $f$ is said to be $(M_f,\theta)$-generalized self-concordant with the parameter $M_f \geq 0$, and order $\theta > 0$ if 
\vspace{-0.75ex}
\begin{equation}\label{eq:gsc_inequality}
\vert \iprods{\nabla^3{f}(x)[v]u, u}\vert \leq M_f\norms{u}_x^2\norms{v}_x^{\theta-2}\norms{v}_2^{3-\theta},~x\in\dom{f},~u, v\in\R^p,
\vspace{-0.75ex}
\end{equation}
where we use the convention $\frac{0}{0} = 0$ for the case $\theta < 2$ and $\theta > 3$.
If $\theta = 3$, then $f$ reduces to the self-concordant function defined by Nesterov and Nemirovskii in \cite{Nesterov1994}.
If $\theta = 3$ and $M_f = 2$, then $f$ is said to be standard self-concordant.
\end{definition}

\beforepara
\paragraph{\textbf{Basic properties:}}
Basic and fundamental properties as well as examples of generalized self-concordant functions can be found in \cite{SunTran2017gsc}.
We recall the following Legendre conjugate of a generalized self-concordant function.
Let $f : \dom{f}\to\R$ be an $(M_f,\theta)$-generalized self-concordant function, we define
\begin{equation}\label{eq:L_conjugate}
f^{\ast}(y) :=  \sup_{x\in\dom{f}}\set{ -y^{\top}x - f(x)},
\end{equation}
the Legendre conjugate of $f$ (i.e. $f^{\ast}(-y)$ is the Fenchel conjugate of $f$).
For generalized self-concordant functions and their conjugates, we have the following result. 

\begin{proposition}\label{pro:conjugate}
\begin{itemize}
\item[$\mathrm{(a)}$]
If $f$ is $(M_f,\theta)$-generalized self-concordant with $\theta \in (0, 3)$ and $\mu_f$-strongly convex w.r.t. the Euclidean norm $\norms{\cdot}_2$, then $f$ is $\hat{M}_f$-self-concordant with $\hat{M}_f := \mu_f^{\frac{\theta-3}{2}}M_f$.
\item[$\mathrm{(b)}$]
If $f$ is an $(M_f,\theta)$-generalized self-concordant function with $\theta \in [3, 6)$, then its Legendre conjugate $f^{\ast}(\cdot)$ is also $(M_f,\theta_{\ast})$-generalized self-concordant with $\theta_{\ast} := 6 - \theta$.
\item[$\mathrm{(c)}$]
Assume that $f$ is  $M_f$-self-concordant on $\dom{f}$  and $g$ is nonlinear and $(M_g,\theta)$-generalized self-concordant on $\dom{g}$ with $\theta\in (0, 3]$.
If $\dom{f}\cap\dom{g}$ is nonempty, closed, and bounded, then $h := f + g$ is  $M_h$-self-concordant with $M_h := \max\set{M_f, \hat{M}_g}$, where $\hat{M}_g := \mu_g^{\frac{\theta-3}{2}}M_g$ and
\begin{equation}\label{eq:mu_g}
\mu_g := \min\set{ \lambda_{\min}( \nabla^2{g}(x)) \mid x\in\dom{f}\cap\dom{g} } \in (0, +\infty),
\end{equation}
if $\theta < 3$, and $\hat{M}_g := M_g$ if $\theta = 3$.
\end{itemize}
\end{proposition}

\begin{proof}
The proof of statements (a) and (b) can be found in \cite[Propositions 4 and 6]{SunTran2017gsc}.
If $\dom{f}\cap\dom{g}$ is nonempty, closed, and bounded, then $g$ is also $\mu_g$-strongly convex on $\dom{f}\cap\dom{g}$ with $\mu_g$ defined by \eqref{eq:mu_g}.
Applying statement (a) to the strongly convex function $g$, we obtain statement (c).
\Eproof
\end{proof}
\beforepara
\paragraph{\textbf{Discussion:}}
Proposition~\ref{pro:conjugate} shows that the class of self-concordant functions can be extended to cover at least three classes of smooth convex functions.
The first one is the class of smooth and strongly convex functions that is also generalized self-concordant as studied in \cite{SunTran2017gsc}.
In the case it is not strongly convex, one can add a small quadratic regularizer to obtain this property.
The second class is the conjugate of generalized self-concordant functions with Lipschitz continuous gradient.
The third class of functions is generalized self-concordant functions on bounded domain.
We believe that these three classes of functions cover a sufficiently large class of applications, see \cite{SunTran2017gsc} for more detailed examples and additional properties.

\beforepara
\paragraph{\textbf{Standard self-concordant barriers:}}
Next, we recall the class of standard self-concordant barriers, and its properties.

\begin{definition}\label{de:scb_function}
Given a nonempty, closed, and convex set $\Kc$ in $\R^p$, we say that $f$ is a $\nu_f$-self-concordant barrier of $\Kc$ if $f$ is standard self-concordant on $\dom{f} \equiv\intx{\Kc}$, $f(x)\to+\infty$ as $x$ approaches the boundary $\partial{\Kc}$ of $\Kc$, and 
\begin{equation}\label{eq:scb_function}
\sup_{u\in\R^p}\set{ \nabla{f}(x)^{\top}u - \norms{u}_x^2} \leq \nu_f,~~\forall x\in\dom{f}.
\end{equation}
The self-concordant barrier $f$ is said to be a logarithmically homogeneous self-concordant barrier if $f(\tau x) = f(x) - \nu_f\ln(\tau)$ for any $\tau > 0$ and $x\in\dom{f}$.
\end{definition}

Given a self-concordant barrier of $\Kc$, we define $x_f^{\star} := \argmin_{x\in\Kc}f(x)$ the analytical center of $\Kc$ if $x_f^{\star}$ exists.
Clearly, if $\Kc$ is bounded, then $x_f^{\star}$ exists.
In addition to these properties, we also have $\norms{x - x^{\star}_f}_{x^{\star}_f} \leq \rho_f$ for any $x\in\dom{f}$, where $\rho_f := \nu_f + 2\sqrt{\nu_f}$ for general self-concordant barrier $f$ and $\rho_f := \nu_f$ if $f$ is logarithmically homogeneous.

\beforesec
\section{Barrier smoothing technique and inexact oracles}\label{sec:barrier_smoothing}
\aftersec
In this section, we describe a barrier smoothing technique for \eqref{eq:constr_cvx} which has been used in  \cite{Kojima1993,Necoara2009,Nesterov2011c,TranDinh2012c,Zhao2005}.
Without loss of generality, we can assume that $M_g = 2$, since any self-concordant function $g$ with the parameter $M_g > 0$, $(M_g^2/4)g$ is standard self-concordant.

\beforesubsec
\subsection{\bf Smoothed dual problem}\label{subsec:smooth_dual_problem}
\aftersubsec
Under Assumption~\ref{as:A1}, we consider the following self-concordant barrier smoothed dual problem of \eqref{eq:constr_cvx} (shortly, smoothed dual problem):
\vspace{-0.75ex}
\begin{equation}\label{eq:smoothed_dual_prob0}
{\!\!\!\!\!}
\overline{D}_t^{\star} := {\!\!} \min_{y\in\R^n}\Big\{ \overline{D}_t(y) := {\!\!\!} \underbrace{\max_{x\in\intx{\Kc}} {\!\!} \Big\{ y^{\top}Ax  - g(x) - t f(x) \Big\}}_{\bar{d}_t(y)}  + \underbrace{\phi^{\ast}(-y)}_{\bar{h}(y)} \Big\}.
{\!\!\!\!\!}
\vspace{-0.75ex}
\end{equation}
Note that $g(\cdot) + tf(\cdot)$ is self-concordant with the parameter $M_t := \max\set{2, \frac{2}{\sqrt{t}}}$ on $\dom{f}\cap\dom{g}$. To make it standard self-concordant, we rescale \eqref{eq:smoothed_dual_prob0} as follows:
\vspace{-0.75ex}
\begin{equation}\label{eq:smoothed_dual_prob}
D_t^{\star} := \min_{y\in\R^n}\Bigg\{D_t(y) := \underbrace{\tfrac{M_t^2}{4}\bar{d}_t(y)}_{d_t(y)}  {~} + {~} \underbrace{\tfrac{M_t^2}{4}\bar{h}(y)}_{h_t(y)}  \Bigg\}.
\vspace{-0.75ex}
\end{equation}
From \cite{Necoara2009} or \cite{Zhao2005}, $d_t$ is standard self-concordant. Clearly, if $t \in (0, 1]$, then $M_t = \frac{2}{\sqrt{t}}$.
In this case, we have $\bar{d}_t(y) = td_t(y)$ and $\bar{h}(y) = th_t(y)$.

To evaluate the (normalized) smoothed dual function $d_t$ and its derivative, we consider the following standard self-concordant function:
\vspace{-0.75ex}
\begin{equation}\label{eq:define_psi}
\psi_t(x;y) := \frac{M_t^2}{4}\big[g(x) + t f(x) - y^{\top}Ax\big].
\vspace{-0.75ex}
\end{equation}

\beforepara
\paragraph{\textbf{Primal local norms:}}
Note that $\nabla^2{\psi_t}(x;y) =  \tfrac{M_t^2}{4}\big[\nabla^2{g}(x) + t\nabla^2{f}(x)\big] = \nabla^2{\psi_t}(x)$ is symmetric positive definite on $\dom{g}\cap\dom{f}$ and independent of $y$.
Therefore, we define the following local norms on the primal space:
\begin{equation}\label{eq:x_local_norms}
\abss{u}_{x, t} := \left(u^{\top}\nabla^2{\psi_t}(x)u\right)^{1/2},~~~\text{and}~~~~\abss{v}_{x,t}^{\ast} := \left(v^{\top}\nabla^2{\psi_t}(x)^{-1}v\right)^{1/2},
\end{equation}
for any $u, v\in\R^p$. 
If $t \in (0, 1]$, then $\abss{u}_{x, t} = \left(u^{\top}\big[\nabla^2{f}(x) + \frac{1}{t}\nabla^2{g}(x)\big]u\right)^{1/2}$.

\beforepara
\paragraph{\textbf{Exact oracles of the dual function $d_t$:}}
We can summarize the properties of $d_t$ defined in \eqref{eq:smoothed_dual_prob} into the following proposition which we omit the proof.

\begin{proposition}\label{pro:properties_of_dt}
Under Assumption~\ref{as:A1}, $\psi_t(\cdot;y)$ defined by \eqref{eq:define_psi} and $d_t(\cdot)$ defined by \eqref{eq:smoothed_dual_prob} are standard self-concordant.
Moreover, if the following  primal subproblem has optimal solution
\begin{equation}\label{eq:barrier_subprob}
x^{\ast}_t(y) := \mathrm{arg}\!\!\!\!\!\min_{x\in\intx{\Kc}}\Big\{ \psi_t(x; y) := \tfrac{M_t^2}{4}\big[ g(x) + tf(x) - y^{\top}Ax \big]\Big\},
\end{equation}
then its solution is unique.
The optimality condition of this subproblem is 
\begin{equation}\label{eq:opt_cond_subprob}
\nabla{\psi}_t(x^{\ast}_t(y); y) \equiv  \tfrac{M_t^2}{4}\big[\nabla{g}(x^{\ast}_t(y)) + t\nabla{f}(x^{\ast}_t(y)) - A^{\top}y\big] = 0,
\end{equation}
which is necessary and sufficient for $x^{\ast}_t(y)$ to be optimal to \eqref{eq:barrier_subprob}.
The function value and derivatives of  $d_t$ in \eqref{eq:smoothed_dual_prob} can be evaluated as (see \cite{Nesterov1994})
\begin{equation}\label{eq:dt_and_derivatives}
\textbf{$($Exact oracles$)$:}~~ \left\{\begin{array}{ll}
d_t(y) & = -\psi_t(x^{\ast}_t(y); y), \vspace{1ex}\\
\nabla{d_t}(y) &=  \tfrac{M_t^2}{4}Ax^{\ast}_t(y), \vspace{1ex}\\
\nabla^2{d_t}(y) &= \frac{M_t^4}{16}A\nabla^2{\psi_t}(x^{\ast}_t(y))^{-1}A^{\top}.
\end{array}\right.
\end{equation}
\end{proposition}
\beforepara
\paragraph{\textbf{Dual local norms:}}
Since $\nabla^2{d_t}(y) \succ 0$,  we can define the following local norms in the dual space:
\begin{equation}\label{eq:dual_local_norm}
\norms{u}_{y,t}  := \big(u^{\top}{\nabla}^2{d_t}(y)u\big)^{1/2} ~~~\text{and}~~~\norms{v}_{y,t}^{\ast} := \big(v^{\top}{\nabla}^2{d_t}(y)^{-1}v\big)^{1/2}.
\end{equation}

\beforesubsec
\subsection{\bf Inexact oracles of the smoothed dual function}\label{subsec:inexact_oracles}
\aftersubsec
When $g$ and $\Kc$ are not trivial, solving the smoothed slave subproblem \eqref{eq:barrier_subprob} exactly is impractical.
We can only approximately solve  \eqref{eq:barrier_subprob} or \eqref{eq:opt_cond_subprob} up to a given accuracy as defined in the following.

\begin{definition}\label{de:inexact_sol_xt}
Let $x^{\ast}_t(y)$ be the exact solution of \eqref{eq:barrier_subprob} at $y \in\R^n$.
We call $\widetilde{x}^{\ast}_t(y)$ a $\delta$-(approximate) solution of \eqref{eq:barrier_subprob}  if $\delta_t(y) := \abss{\widetilde{x}^{\ast}_t(y) - x^{\ast}_t(y)}_{\widetilde{x}^{\ast}_t(y), t} \leq \delta$, where $\abss{\cdot}_{x,t}$ is defined by \eqref{eq:x_local_norms}.
\end{definition}

Given an inexact solution $\widetilde{x}^{\ast}_t(y)$ of \eqref{eq:barrier_subprob} as defined in Definition~\ref{de:inexact_sol_xt}, we define an inexact oracle of $d_t$ as follows:
\begin{equation}\label{eq:inexact_oracle}
\textbf{$($Inexact oracles$)$:}~~\left\{\begin{array}{ll}
\widetilde{d}_t(y) &= -\psi_t(\widetilde{x}^{\ast}_t(y); y) , \vspace{1ex}\\
\widetilde{\nabla}{d_t}(y) &= \tfrac{M_t^2}{4} A\widetilde{x}^{\ast}_t(y), \vspace{1ex}\\
\widetilde{\nabla}^2{d_t}(y) &= \frac{M_t^4}{16}A\nabla^2{\psi_t}(\widetilde{x}^{\ast}_t(y))^{-1}A^{\top}.
\end{array}\right.
\end{equation}
Since $\nabla^2\psi_t(\cdot)$ is positive definite and $A$ is full-row rank, $\widetilde{\nabla}^2{d_t}(y)$ is positive definite.
Now we define the following local norms using inexact oracles:
\begin{equation}\label{eq:inexact_dual_local_norm}
\tnorms{u}_{y,t}  := \big(u^{\top}\widetilde{\nabla}^2{d_t}(y)u\big)^{1/2} ~~~~\text{and}~~~~\tnorms{v}_{y,t}^{\ast} := \big(v^{\top}\widetilde{\nabla}^2{d_t}(y)^{-1}v\big)^{1/2}.
\end{equation}

We first prove some properties of inexact solution $\widetilde{x}^{\ast}_t(y)$ and inexact oracles of $d_t$ defined by \eqref{eq:inexact_oracle} in the following proposition, whose proof can be found in Appendix~\ref{apdx:pro:inexact_oracle}.

\begin{proposition}\label{pro:inexact_oracle}
For any $\delta \in [0, 1)$, we have:
\begin{equation}\label{eq:approx_cond}
\textrm{if}~~\abss{\nabla\psi_t({\widetilde{x}^{\ast}_t(y)};y)}_{\widetilde{x}^{\ast}_t(y), t}^{\ast}  \leq \tfrac{\delta}{1 + \delta} ~~\text{then}~~\delta_t(y) := \abss{\widetilde{x}^{\ast}_t(y) - x^{\ast}_t(y)}_{\widetilde{x}^{\ast}_t(y), t} \leq \delta.
\end{equation}

In addition, for $d_t$ and its derivatives defined by \eqref{eq:dt_and_derivatives}, and its inexact oracle defined by \eqref{eq:inexact_oracle}, the following properties hold
\begin{equation}\label{eq:inexact_oracle_properties}
\left\{\begin{array}{ll}
&0 \leq \omega\left(\frac{\delta_t(y)}{1+\delta_t(y)}\right) \leq d_t(y) - \widetilde{d}_t(y) \leq \omega_{\ast}\left(\frac{\delta_t(y)}{1-\delta_t(y)}\right), \vspace{1.25ex}\\
&\big(1 -\delta_t(y)\big)^2\widetilde{\nabla}^2{d_t}(y) \preceq \nabla^2{d_t}(y) \preceq \big(1 - \delta_t(y)\big)^{-2}\widetilde{\nabla}^2{d_t}(y), \vspace{1.25ex}\\
&\tnorms{\widetilde{\nabla}{d_t}(y) - \nabla{d_t}(y)}_{y,t}^{\ast} \leq \delta_t(y),
\end{array}\right.
\end{equation}
where $\omega(\tau) := \tau - \ln(1+\tau)$ for $\tau \geq 0$ and $\omega_{\ast}(\tau) := -\tau - \ln(1 - \tau)$ for $\tau \in [0, 1)$.
\end{proposition}

\noindent\textbf{Discussion:}
The first estimate \eqref{eq:approx_cond} shows that to obtain an approximate solution $\widetilde{x}^{\ast}_t(y)$ such that $\abss{\widetilde{x}^{\ast}_t(y) - x^{\ast}_t(y)}_{\widetilde{x}^{\ast}_t(y), t} \leq \delta$, we need to solve the slave problem \eqref{eq:barrier_subprob} such that
\begin{equation}\label{eq:primal_stopping}
\abss{\nabla{g}(\widetilde{x}^{\ast}_t(y)) + t\nabla{f}(\widetilde{x}^{\ast}_t(y)) - A^{\top}y}_{\widetilde{x}^{\ast}_t(y), t}^{\ast} \leq \frac{4\delta}{M_t^2(1+\delta)}.
\end{equation}
This condition is implementable, e.g., when we apply a Newton-type method to solve the nonlinear system \eqref{eq:opt_cond_subprob}.
The estimates in \eqref{eq:inexact_oracle_properties} show us how the inexact oracles in \eqref{eq:inexact_oracle} approximate the exact ones in \eqref{eq:dt_and_derivatives}.

\beforepara
\paragraph{\textbf{Approximate primal-dual solutions:}}
Given an accuracy $\varepsilon > 0$, our goal is to compute an $\varepsilon$-approximate primal-dual solution $(\tilde{x}^{\star}, \tilde{y}^{\star})$ to $(x^{\star}, y^{\star})$ of \eqref{eq:opt_cond} in the following sense:

\begin{definition}\label{de:primal_dual_approx_sols}
A pair $(\tilde{x}^{\star}, \tilde{y}^{\star})$  is called an $\varepsilon$-approximate primal-dual solution to an exact primal-dual one $(x^{\star}, y^{\star})$ of \eqref{eq:opt_cond} if 
\begin{equation}\label{eq:approx_opt_cond}
\left\{\begin{array}{lll}
&\abss{ A^{\top}\tilde{y}^{\star} - \nabla{g}(\tilde{x}^{\star})}_{\tilde{x}^{\star}, t}^{\ast} \leq \varepsilon ~~~~~~~&\text{($\varepsilon$-primal optimality)}, \vspace{1ex}\\
&r \in \tilde{y}^{\star} + \partial{\phi}(A\tilde{x}^{\star} + e) ~~~~~&\text{($\varepsilon$-dual  optimality)},\vspace{1ex}\\
&\tilde{x}^{\star} \in \intx{\Kc}  ~~~~~~&\text{(primal feasibility)}, \vspace{1ex}\\
&~~\tnorms{e}_{\tilde{y}^{\star}, t}^{\ast} \leq \varepsilon~~~\text{and}~~~\abss{A^{\top}r}_{\tilde{x}^{\star},t}^{\ast} \leq \varepsilon.
\end{array}\right.
\end{equation}
\end{definition}
Here, the errors are measured through local norms in primal and dual spaces defined in \eqref{eq:x_local_norms} and \eqref{eq:inexact_dual_local_norm}.
These norms are computable since they are defined through $\tilde{x}^{\star}$ and $\tilde{y}^{\star}$.
In addition, since $A$ is full-row rank, $A^{\top}r = 0$ if and only if $r = 0$.
Because $\tilde{x}^{\star} \in \intx{\Kc}$, we have $\Nc_{\Kc}(\tilde{x}^{\star}) = \set{\boldsymbol{0}}$. Therefore the first line of \eqref{eq:approx_opt_cond} can approximate the first line of \eqref{eq:opt_cond}.
Similarly, the second line of \eqref{eq:approx_opt_cond} approximates the second line of \eqref{eq:opt_cond}, i.e. $0 \in  y^{\star} + \partial{\phi}(Ax^{\star})$. 
Therefore, Definition~\ref{de:primal_dual_approx_sols} is consistent with the optimality condition~\eqref{eq:opt_cond}.

\beforesec
\section{Inexact IPLD Method with Inexact Oracles}\label{sec:dual_decomp_alg}
\aftersec
We develop an inexact interior-point Lagrangian decomposition method to solve \eqref{eq:constr_cvx} by using the inexact oracles \eqref{eq:inexact_oracle}.

\beforesubsec
\subsection{\bf Inexact proximal-Newton method for \eqref{eq:smoothed_dual_prob}}\label{subsec:inexact_newton_dual}
\aftersubsec
\noindent\textbf{The optimality condition of \eqref{eq:smoothed_dual_prob}:}
Recall the smoothed dual problem \eqref{eq:smoothed_dual_prob}, its optimality condition is
\begin{equation}\label{eq:dual_opt_cond}
0\in\nabla{d_t}(y) + \partial h_t(y) = \frac{M_t^2}{4}Ax^{\ast}_t(y)  +\partial h_t(y).
\end{equation}
Any $y_t^{\ast}$ satisfies \eqref{eq:dual_opt_cond} is an optimal solution of \eqref{eq:smoothed_dual_prob}.
The sequence $\set{ (x^{\ast}_{t}(y^{\ast}_t), y^{\ast}_t)}_{t\geq 0}$ forms a central path, which converges to $(x^{\star}, y^{\star})$ a primal-dual solution of \eqref{eq:constr_cvx}.

\noindent\textbf{Exact Proximal Newton scheme:}
Suppose that we are currently at $y^k$, since $d_t$ is twice differentiable, we will apply proximal-Newton method to compute $\bar{y}^{k+1}$, which leads to 
\begin{equation}\label{eq:dual_newton_scheme} 
0 \in \widetilde{\nabla}^2{d_{t_{k+1}}}(y^k)(\bar{y}^{k+1} - y^k)  + \widetilde{\nabla}{d_{t_{k+1}}}(y^k) + \partial{h_{t_{k+1}}}(\bar{y}^{k+1}).
\end{equation}
If we define
\begin{equation}\label{eq:Pk_t}
Q_{t_{k+1}}(y) := \iprods{\widetilde\nabla{d_{t_{k+1}}}(y^k), y - y^k} +  \frac{1}{2} \iprods{\widetilde\nabla^2{d_{t_{k+1}}}(y^k)(y - y^k), y - y^k} + h_{t_{k+1}}(y),
\end{equation}
then we can write $\bar{y}^{k+1} := \argmin_{y}Q_{t_{k+1}}(y)$.
Introducing the notation $\prox_{h_t}^{\widetilde\nabla^2{d_t}}(\cdot)$, we can write \eqref{eq:dual_newton_scheme} in the following form (see \cite{TranDinh2016c} for a concrete definition)
\begin{equation}\label{eq:dual_newton_scheme2}
\bar{y}^{k+1} := \prox_{h_{t_{k+1}}}^{\widetilde{\nabla}^2{d_{t_{k+1}}}(y^k)}\left( y^k - \widetilde{\nabla}^2{d_{t_{k+1}}}(y^k)^{-1}\widetilde{\nabla}{d_{t_{k+1}}}(y^k)\right).
\end{equation}
\noindent\textbf{Inexact Proximal Newton scheme:}
Similarly, we can also approximately solve \eqref{eq:dual_newton_scheme} up to a given accuracy as.
\begin{equation}\label{eq:inexact_dual_newton_scheme2}
y^{k+1} :\approx \prox_{h_{t_{k+1}}}^{\widetilde{\nabla}^2{d_{t_{k+1}}}(y^k)}\left( y^k - \widetilde{\nabla}^2{d_{t_{k+1}}}(y^k)^{-1}\widetilde{\nabla}{d_{t_{k+1}}}(y^k)\right). 
\end{equation}
Here, the approximation ``$:\approx$'' is defined in the following sense:
\begin{definition}\label{de:inexact_sol}
For a given $\epsilon \geq 0$ and $Q_{t_{k+1}}$ defined by \eqref{eq:Pk_t}, a vector $y^{k+1}$ given in \eqref{eq:inexact_dual_newton_scheme2} is said to be an $\epsilon$-approximate solution to  $\bar{y}^{k+1}$ of \eqref{eq:dual_newton_scheme} if
\begin{equation}\label{eq:inexact_yk}
Q_{t_{k+1}}(y^{k+1}) - Q_{t_{k+1}}(\bar{y}^{k+1}) \leq \frac{\epsilon^2}{2}.
\end{equation}
\end{definition}
Note that \eqref{eq:inexact_yk} implies $\tnorm{y^{k+1} - \bar{y}^{k+1}}_{y^k, t_{k+1}} \leq \epsilon$. 
There exists several convex optimization methods to compute $y^{k+1}$ in \eqref{eq:inexact_dual_newton_scheme2}.
For example, we can apply accelerated proximal gradient methods such as FISTA \cite{Beck2009,Nesterov2007} to compute this point.
We can also apply semi-smooth Newton-CG augmented Lagrangian methods in \cite{li2018highly,zhao2010newton} to solve this problem.
We will discuss the computation of $y^{k+1}$ in detail in Section~\ref{sec:numerical_experiment}.

\noindent\textbf{Generalized  gradient mapping:}
Now let us define the following inexact generalized gradient mapping
\begin{equation}\label{eq:gradient_mapping}
\widetilde{G}_{t}(y) := \widetilde{\nabla}^2{d_t}(y)\left( y - \text{prox}_{h_t}^{\widetilde{\nabla}^2{d_t}(y)}(y-\widetilde{\nabla}^2{d_t}(y)^{-1}\widetilde{\nabla}{d_t}(y)\right).
\end{equation}
Using $\widetilde{\nabla}^2{d_t}(\cdot)$ defined by \eqref{eq:inexact_oracle}, we further define the following quantity:
\begin{equation}\label{eq:inexact_NT_decrement}
\begin{array}{ll}
\lambda_{t}(y) &:= \tnorms{\widetilde{G}_t(y)}_{y, t}^{\ast} = \iprods{\widetilde{\nabla}^2d_{t}(y)^{-1}\widetilde{G}_t(y),\widetilde{G}_t(y)}^{1/2}.
\end{array}
\end{equation}
We call $\lambda_{t}(y)$ the inexact proximal-Newton decrement.
In Subsection \ref{subsec:convergence_analysis} we can show that this quantity can be used to characterize the optimality condition \eqref{eq:opt_cond}.

\beforesubsec
\subsection{\bf The algorithm}\label{subsec:the_algorithm}
\aftersubsec
From the above analysis, we can combine all the steps together and describe an algorithm to solve \eqref{eq:constr_cvx} as in  Algorithm~\ref{alg:A1}.
In this algorithm, we explicitly show how to choose the accuracy of inexact oracles and inexact proximal-Newton direction, and how to update the penalty parameter $t$.

\begin{algorithm}[ht!]\caption{(\textit{Inexact Interior-Point Lagrangian Decomposition Algorithm})}\label{alg:A1}
\begin{algorithmic}[1]
\normalsize
\STATE\textbf{Phase 1: Find an initial point.}\label{step:A1_step0}~Given any value $t_0 \in (0, 1]$ and $\beta \in (0, \frac{1}{10}]$, find starting points $y^0\in\R^n$ and $x^0\in\R^p$  such that 
\begin{equation}\label{eq:initial_point_cond}
\tnorms{\widetilde{G}_{t_0}(y^0)}_{y^0, t_0}^{\ast} \leq \beta~~~~\text{and}~~~~\abss{\nabla\psi_{t_{0}}(x^0; y^0)}_{x^0, t_{0}}^{\ast} \leq \tfrac{\tilde{\delta}_0}{1+ \tilde{\delta}_0},
\end{equation}
by using Algorithm~\ref{alg:A2} below, for any predefined accuracy $\tilde{\delta}_0 \in (0, \frac{\beta}{100}]$.
\STATE\textbf{Phase 2: Main iteration.}~\textrm{For $k = 0$ to $k_{\max}$, perform}
\vspace{0.75ex}
\STATE\hspace{0.2cm}\label{step:A1_step1} Update $t_k$ as $t_{k+1} := \sigma t_k$, where $\sigma \in (0, 1)$ is defined by \eqref{eq:sigma_cond} below.
\vspace{0.75ex}
\STATE\hspace{0.2cm}\label{step:A1_step2} Solve approximately \eqref{eq:opt_cond_subprob} at $y = y^k$ up to an accuracy $\delta_k \in (0, \frac{\beta}{100}]$ to get $x^{k+1} := \tilde{x}_{t_{k+1}}^{\ast}(y^k)$, i.e.:
\vspace{-1ex}
\begin{equation*}
\abss{\nabla\psi_{t_{k+1}}(x^{k+1}; y^k)}_{x^{k+1}, t_{k+1}}^{\ast} \leq \tfrac{\delta_k}{1+\delta_k}.
\vspace{-2ex}
\end{equation*}
\STATE\hspace{0.2cm}\label{step:A1_step3}(\textbf{Inexact oracles}): Evaluate inexact gradient and Hessian of $d_t$ as
\vspace{-1ex}
\begin{equation}\label{eq:inexact_oracle_k}
\left\{\begin{array}{ll}
\widetilde{\nabla}{d_{t_{k+1}}}(y^k) &:= \frac{M_{t_{k+1}}^2}{4}Ax^{k+1} ,\vspace{1ex}\\
\widetilde{\nabla}^2{d_{t_{k+1}}}(y^k) &:= \frac{M_{t_{k+1}}^4}{16}A\nabla^2{\psi_{t_{k+1}}}(x^{k+1})^{-1}A^{\top}.
\end{array}\right.
\vspace{-1ex}
\end{equation}
\STATE\hspace{0.2cm}\label{step:A1_step4}(\textbf{Inexact proximal-Newton step}):
Compute $y^{k+1}$ up to an accuracy $\epsilon_k \in (0, \frac{\beta}{100}]$,  i.e.:
\vspace{-1ex}
\begin{equation*}
y^{k+1} :\approx \text{prox}_{h_{	t_{k+1}}}^{\widetilde{\nabla}^2d_{t_{k+1}}(y^k)}\left(y^k-\widetilde{\nabla}^2d_{t_{k+1}}(y^k)^{-1}\widetilde{\nabla}d_{t_{k+1}}(y^k) \right).
\vspace{-2ex}
\end{equation*}
\STATE\hspace{0cm}\textbf{End.}
\end{algorithmic}
\end{algorithm}

Note that we have not specified how to find a starting point $(x^0, y^0)$ to guarantee \eqref{eq:initial_point_cond}  and how to set $k_{\max}$ in Algorithm~\ref{alg:A1}. 
In Subsection \ref{subsec:initial_point}, we will show that such an $(x^0, y^0)$ can be found in finite steps. 
In Subsection \ref{subsec:convergence_analysis}, we show how to set $k_{\max}$ to get an $\varepsilon$-approximate primal-dual solution of \eqref{eq:constr_cvx}. 
\beforesubsec
\subsection{\bf Convergence analysis}\label{subsec:one_iter_analysis}
\aftersubsec
Our analysis consists of several steps and is organized as follows:
\begin{itemize}
\vspace{-1.5ex}
\item Lemma~\ref{le:key_property0} provides an estimate between $\lambda_{t_{k+1}}(y^{k+1})$ and $\lambda_{t_{k+1}}(y^{k})$ in \eqref{eq:inexact_NT_decrement}.
\item Lemma~\ref{le:key_property1} bounds $\lambda_{t_{k+1}}(y^k)$ in terms of  $\tilde{\Delta}_{t_k}$, $\tilde{\Delta}_{t_{k+1}}$ and $\lambda_{t_k}(y^k)$, where $\tilde{\Delta}_{t_k}$ and $\tilde{\Delta}_{t_{k+1}}$ measure the distances between $\tilde{x}_{t_{k}}^{\ast}(y^k)$ and $\tilde{x}_{t_{k+1}}^{\ast}(y^k)$.
\item Lemma~\ref{le:key_property2} shows how to upper bound $\tilde{\Delta}_{t_k}$ and $\tilde{\Delta}_{t_{k+1}}$.
\item The main result of this section is Theorem~\ref{thm:one_step_analysis} which provides an update rule of $t$ to maintain the point $y^k$ in the neighborhood of the central path.
The proof of this theorem is obtained by combining all the above  lemmas.
\vspace{-1.25ex}
\end{itemize}
Firstly, we state the main estimate of the inexact Newton-type step at Step~\ref{step:A1_step4} of Algorithm~\ref{alg:A1} in Lemma~\ref{le:key_property0}, whose proof is given in Appendix~\ref{apdx:le:key_property0}.


\begin{lemma}\label{le:key_property0}
Let $\set{y^{k}}$ be generated by Algorithm \ref{alg:A1}, and 
\begin{equation*}
\abss{\widetilde{x}^{\ast}_{t_{k+1}}(y^{k+1}) - x^{\ast}_{t_{k+1}}(y^{k+1})}_{\widetilde{x}^{\ast}_{t_{k+1}}(y^{k+1}), t_{k+1}} \leq \tilde{\delta}_{k+1}.
\end{equation*}
Then
\begin{equation}\label{eq:key_est00}
{\!\!\!}\begin{array}{ll}
\lambda_{t_{k+1}}(y^{k+1}) &\leq  \tilde{\delta}_{k+1} +  \frac{1}{\big(1{~} - {~}\tilde{\delta}_{k+1}\big)\big( 1 -  \delta_k - \lambda_{t_{k+1}}(y^k) - \epsilon_k\big)}\Bigg[3\epsilon_k {~} + {~} \delta_k \vspace{1ex}\\
& +  {~} \sqrt{4\delta_k {~} - {~} 2\delta_k^2}\big(\lambda_{t_{k+1}}(y^k) {~} + {~} \epsilon_k\big) +  \frac{\big(\lambda_{t_{k+1}}(y^k) + {~} \epsilon_k\big)^2}{\big(1 - \lambda_{t_{k+1}}(y^k) -  \delta_k  -  \epsilon_k\big)}\Bigg].
\end{array}{\!\!\!}
\end{equation}
In particular, if $\tilde{\delta}_{k+1} = 0$, $\delta_k = 0$, and $\epsilon_k = 0$, then \eqref{eq:key_est00} reduces to
\begin{equation}\label{eq:key_est01}
\lambda_{t_{k+1}}(y^{k+1}) \leq \frac{\lambda_{t_{k+1}}(y^k)^2}{(1 - \lambda_{t_{k+1}}(y^k))^2}.
\end{equation}
\end{lemma}

Note that if we solve both the slave problem at Step~\ref{step:A1_step2} and the master problem at Step~\ref{step:A1_step4} exactly, then we  could obtain the estimate \eqref{eq:key_est01}, which is the same as in standard interior-point path-following methods \cite{Nesterov2004}.
Next, we show a relation between $\lambda_{t_{k+1}}(y^k)$ and $\lambda_{t_{k}}(y^k)$, whose proof is in Appendix~\ref{apdx:le:key_property1}.

\begin{lemma}\label{le:key_property1}
Let $t_k$ be updated as $t_{k+1} := \sigma t_k$ for given $\sigma\in (0, 1)$.
Define
\begin{equation}\label{eq:Delta_quantity}
\left\{\begin{array}{ll}
\tilde{\Delta}_{t_k} & := \abss{\widetilde{x}^{\ast}_{t_{k+1}}(y^k) - \widetilde{x}^{\ast}_{t_k}(y^k)}_{\widetilde{x}^{\ast}_{t_k}(y^k), t_k},  \vspace{1ex}\\
\tilde{\Delta}_{t_{k+1}} &:= \abss{\widetilde{x}^{\ast}_{t_{k+1}}(y^k) - \widetilde{x}^{\ast}_{t_k}(y^k)}_{\widetilde{x}^{\ast}_{t_{k+1}}(y^k), t_{k+1}}.
\end{array}\right.
\end{equation}
Then, the following estimate holds 
\begin{equation}\label{eq:key_est02}
\lambda_{t_{k+1}}(y^k) \leq \tilde{\Delta}_{t_{k+1}} + \Bigg[\frac{1 + \sqrt{1 - 2\sigma(1-\tilde{\Delta}_{t_k})^2 + \sigma}}{\sigma(1-\tilde{\Delta}_{t_k})}\Bigg]\lambda_{t_k}(y^k).
\end{equation}
\end{lemma}

The following lemma shows how to bound $\tilde{\Delta}_{t_k}$ and $\tilde{\Delta}_{t_{k+1}}$, the distances between $\widetilde{x}^{\ast}_{t_k}(y^k)$ and $\widetilde{x}^{\ast}_{t_{k+1}}(y^k)$, whose proof is given in Appendix~\ref{apdx:le:key_property2}.

\begin{lemma}\label{le:key_property2}
Let $\tilde{\Delta}_{t_k}$ and $\tilde{\Delta}_{t_{k+1}}$ be defined by \eqref{eq:Delta_quantity}, and $t_{k+1} := \sigma t_k$ for some $\sigma \in (0, 1)$.
We define the following quantities:
\begin{equation}\label{eq:lm3_def}
\left\{\begin{array}{ll}
\hat{\delta}_{t_k}  & := \abss{\nabla{\psi_{t_k}}(\widetilde{x}^{\ast}_{t_k}(y^k); y^k)}_{\tilde{x}^{\ast}_{t_k}(y^k), t_k}^{\ast}  \vspace{1ex}\\
\hat{\delta}_{t_{k+1}} &:=  \abss{\nabla{\psi_{t_{k+1}}}(\widetilde{x}^{\ast}_{t_{k+1}}(y^k); y^k)}_{\widetilde{x}^{\ast}_{t_{k+1}}(y^k), t_{k+1}}^{\ast}.
 \end{array}\right.
\end{equation}
Then, we have
\begin{equation}\label{eq:key_est03}
\left\{\begin{array}{ll}
\frac{\tilde{\Delta}_{t_k}^2}{1 {~} + {~} \tilde{\Delta}_{t_k}} & \leq \tilde{\Delta}_{t_k}\hat{\delta}_{t_k} + \left(\sigma\hat{\delta}_{t_{k+1}} {~} + {~} (1 - \sigma)\sqrt{\nu_f}\right)\tilde{\Delta}_{t_{k+1}}\vspace{1ex}\\
\frac{\tilde{\Delta}_{t_{k+1}}^2}{1 {~} + {~} \tilde{\Delta}_{t_{k+1}}} &\leq \tilde{\Delta}_{t_{k+1}}\hat{\delta}_{t_{k+1}} {~} + {~} \left(\frac{\hat{\delta}_{t_k}}{\sigma} {~} + {~} \big(\frac{1-\sigma}{\sigma} \big)\sqrt{\nu_f}\right)\tilde{\Delta}_{t_k},
\end{array}\right.
\end{equation}
where $\nu_f$ is the barrier parameter of $f$. In particular, for fixed $\delta \in [0, 1)$, if we choose $\hat{\delta}_{t_k} \leq \delta$ and $\hat{\delta}_{t_{k+1}} \leq \delta$, then
\begin{equation}\label{eq:key_est04}
\left\{ \begin{array}{ll}
\frac{\tilde{\Delta}_{t_k}^2}{1 {~} + {~} \tilde{\Delta}_{t_k}}  &\leq  \delta\cdot \tilde{\Delta}_k {~} + {~} c_{\nu}(\sigma) \cdot \tilde{\Delta}_{t_{k+1}} \vspace{1ex}\\
\frac{\tilde{\Delta}_{t_{k+1}}^2}{1 {~} + {~}  \tilde{\Delta}_{t_{k+1}}} &\leq \delta\cdot \tilde{\Delta}_{t_{k+1}} {~} + {~} c_{\nu}(\sigma) \cdot \tilde{\Delta}_{t_k},
\end{array}\right.
\end{equation}
where $c_{\nu}(\sigma) := \frac{\delta}{\sigma} + \big(\frac{1-\sigma}{\sigma}\big)\sqrt{\nu_f}$ is a decreasing function of $\sigma$ on $(0,1]$.
As a consequence, we also have
\begin{equation}\label{eq:key_est05}
\tilde{\Delta}_{t_k} \leq \frac{\delta {~}+ {~} c_{\nu}(\sigma)}{1 {~} - {~} \delta {~} - {~} c_{\nu}(\sigma)} ~~~~\text{and}~~~~ \tilde{\Delta}_{t_{k+1}} \leq \frac{\delta {~} + {~} c_{\nu}(\sigma)}{1 {~} - {~} \delta {~} - {~} c_{\nu}(\sigma)}.
\end{equation}
\end{lemma}

Utilizing the results of Lemma~\ref{le:key_property0}, Lemma~\ref{le:key_property1} and Lemma~\ref{le:key_property2}, we can prove the following main result on the iteration-complexity of Algorithm~\ref{alg:A1}.


\begin{theorem}\label{thm:one_step_analysis}
Let us choose $\beta \in (0, \frac{1}{10}]$.
Suppose that we choose $\tilde{\delta}_{0}, \delta_k, \tilde{\delta}_{k+1}, \epsilon_k \in [0,  \frac{\beta}{100}]$ and update $t_k$ in  Algorithm \ref{alg:A1} as $t_{k+1} := \sigma t_k$ with
\begin{equation}\label{eq:sigma_cond}
\sigma := 1 - \frac{0.29\beta}{0.3\beta + \sqrt{\nu_f}}  \in (0, 1).
 \end{equation}
In addition, if $y^0\in\R^n$ and $x^0\in\R^p$ satisfy \eqref{eq:initial_point_cond}, then for all $k\geq 0$, we have
\begin{equation*}
\lambda_{t_k}(y^{k}) \leq \beta.
\end{equation*}
Consequently, the number of iterations to obtain $t_k \leq \hat{\varepsilon}$ for a given $\hat{\varepsilon} > 0$ and $\lambda_{t_k}(y^k) \leq \beta$ does not exceed:
\begin{equation}\label{eq:iter_complexity_bound}
k_{\max} := \left\lfloor \frac{\ln\left(\frac{t_0}{\hat{\varepsilon}}\right)}{-\ln(\sigma)} \right\rfloor = \BigO{\sqrt{\nu_f}\ln\left(\frac{t_0}{\hat{\varepsilon}}\right)},
\end{equation}
where $\nu_f$ is the barrier parameter of $f$ and $t_0 \in (0, 1]$.
\end{theorem}

\begin{proof}
Let us first assume that $ \lambda_{t_{k+1}}(y^k) \leq 2.1\beta$. Using $\delta_k, \tilde{\delta}_{k+1}, \epsilon_k \in [0, 10^{-2}\beta]$, after a few elementary calculations, we can overestimate \eqref{eq:key_est00} in Lemma~\ref{le:key_property0} as
\begin{equation}\label{eq:key_est00_new}
\begin{array}{ll}
\lambda_{t_{k+1}}(y^{k+1}) &\leq  \frac{\beta}{100} +  \frac{0.04\beta +0.1\sqrt{4\beta-0.02\beta^2}\left(2.11\beta\right)}{\left(1  - 10^{-2}\beta\right)\left( 1 -  2.12\beta\right)}\vspace{1ex}\\
 & + {~}  \frac{\left(2.11\beta\right)^2}{(1 - 10^{-2}\beta\big)\left( 1 -  2.12\beta\right)^2} \vspace{1ex}\\
 & \leq \beta,
\end{array}
\end{equation}
when $\beta \in (0,\frac{1}{10}]$. Now, we prove that $\lambda_{t_{k+1}}(y^k) \leq 2.1\beta$ is always satisfied. 
Indeed, since
\begin{equation*}
\begin{array}{lll}
\abss{\nabla{\psi_{t_k}}(\widetilde{x}^{\ast}_{t_k}(y^k); y^k)}_{\tilde{x}^{\ast}_{t_k}(y^k), t_k}^{\ast} &\leq \frac{\tilde{\delta}_k}{1+\tilde{\delta}_k} &\leq \frac{\beta}{100} \vspace{1ex}\\
\abss{\nabla{\psi_{t_{k+1}}}(\widetilde{x}^{\ast}_{t_{k+1}}(y^k); y^k)}_{\widetilde{x}^{\ast}_{t_{k+1}}(y^k), t_{k+1}}^{\ast}&\leq \frac{\delta_k}{1+\delta_k} &\leq \frac{\beta}{100},
\end{array}
\end{equation*}
we can choose $\delta$ in Lemma \ref{le:key_property2} to be $\frac{\beta}{100}$. 
In addition, from $\sigma := 1 - \frac{0.29\beta}{0.3\beta + \sqrt{\nu_f}}$, we can show that
\begin{equation*}
c_{\nu}(\sigma) := \frac{\delta}{\sigma} + \frac{1-\sigma}{\sigma}\sqrt{\nu_f} = \frac{10^{-2}\beta}{\sigma} + \frac{1-\sigma}{\sigma}\sqrt{\nu_f} = 0.3\beta.
\end{equation*}
Next, using \eqref{eq:key_est05}, we get
\begin{equation*}
\begin{array}{ll}
\tilde{\Delta}_{t_k} &\leq \frac{0.01\beta + 0.3\beta}{1 - 10^{-2}\beta - 0.3\beta} \leq 0.4493\beta ~~\text{and}~~\tilde{\Delta}_{t_{k+1}} \leq \frac{10^{-2}\beta + 0.3\beta}{1 - 10^{-2}\beta - 0.3\beta} \leq 0.4493\beta.
\end{array}
\end{equation*}
Finally, combining these estimates and \eqref{eq:key_est02} we can show that
\begin{equation*}
\lambda_{t_{k+1}}(y^k) \leq 0.4493\beta + \frac{1 + \sqrt{1 - 2\sigma(1-0.4493\beta)^2 + \sigma}}{\sigma(1-0.4493\beta)}\beta \overset{\tiny\eqref{eq:sigma_cond}}{\leq} 2.1\beta,
\end{equation*}
when $\beta \in (0, \frac{1}{10}]$. Since $t_k := \sigma^kt_0 = \left(1 - \frac{0.29\beta}{0.3\beta + \sqrt{\nu_f}}\right)^kt_0$, to guarantee $t_k \leq \hat{\varepsilon}$, we impose $\sigma^kt_0 = \left(1 - \frac{0.29\beta}{0.3\beta + \sqrt{\nu_f}}\right)^kt_0 \leq \hat{\varepsilon}$.
Note that $-\ln\left(1 - \frac{0.29\beta}{0.3\beta + \sqrt{\nu_f}}\right) \sim \frac{0.29\beta}{0.3\beta + \sqrt{\nu_f}} \sim~\frac{1}{\sqrt{\nu_f}}$, we have
\begin{equation*}
k \geq \left\lfloor \frac{\ln\left(\frac{t_0}{\hat{\varepsilon}}\right)}{-\ln(\sigma)} \right\rfloor = \BigO{\sqrt{\nu_f}\ln\left(\frac{t_0}{\hat{\varepsilon}}\right)},
\end{equation*}
as stated in \eqref{eq:iter_complexity_bound}.
Here, $\sim$ means that two quantities can be approximated by the same order.
\Eproof
\end{proof}

\beforepara
\paragraph{\textbf{The worst-case iteration complexity:}}
Theorem~\ref{thm:one_step_analysis} shows that for any $\hat{\varepsilon} > 0$,  the number of iterations $k$ to obtain $y^k$ such that $\lambda_{t_k}(y^k) \leq \beta$ and $t_k \leq \hat{\varepsilon}$ does not exceed 
\begin{equation*}
\BigO{\sqrt{\nu_f}\ln\left(\frac{t_0}{\hat{\varepsilon}}\right)}, 
\end{equation*}
which is the same as in standard interior-point methods \cite{Nesterov2004,Nesterov1994} up to a constant factor.
It depends on $\sqrt{\nu_f}$, where $\nu_f$ is the barrier parameter of $f$.
Note that the parameter $\beta$ in Algorithm~\ref{alg:A1} represents the radius of the central path neighborhood as in standard path-following methods. 
While the range of $\beta$ in standard exact path-following methods \cite{Nesterov2004} is $(0, \frac{3-\sqrt{5}}{2}]$, it is $[0, \frac{1}{10}]$ in our method.
Clearly, the latter is much smaller than the former one.
However, this range was roughly estimated in our analysis and it is affected by the inexactness in our algorithm.

As we will show in Subsection~\ref{subsec:convergence_analysis}, the conditions $\lambda_{t_k}(y^k) \leq \beta$ and $t_k \leq \hat{\varepsilon}$ imply an approximate solution of \eqref{eq:constr_cvx} and \eqref{eq:dual_prob}.


\beforesubsec
\subsection{\bf Optimality certification}\label{subsec:convergence_analysis}
\aftersubsec
Our goal is to compute an approximate solution of the primal problem \eqref{eq:constr_cvx}.
The following theorem shows how we can find this approximate solution for both the primal and dual problem.

\begin{theorem}\label{th:key_property7}
Let $\sets{(x^{k+1}, y^k)}$ be the sequence generated by Algorithm~\ref{alg:A1}.
Then, for $t_{k+1} \in (0, 1]$ we have the following guarantees:
\begin{equation}\label{eq:lm41_res1}
{\!\!\!\!}\left\{\begin{array}{ll}
&x^{k+1} \in \intx{\Kc}, \vspace{1ex}\\
& \abss{A^{\top}y^k -\nabla{g}(x^{k+1})}_{x^{k+1},t_{k+1}}^{\ast} \leq \left(\sqrt{\nu_f} + \frac{\delta_k}{1+\delta_k}\right)t_{k+1}, \vspace{1.25ex}\\
&r^k \in y^k + \partial{\phi}(Ax^{k+1}+ e^k), \vspace{1ex}\\
&\tnorms{e^k}_{y^k, t_k+1}^{\ast} \leq t_{k+1}\lambda_{t_{k+1}}(y^k),~\text{and}~~\abss{A^{\top}r^k}_{x^{k+1}, t_{k+1}}^{\ast} \leq t_{k+1}\lambda_{t_{k+1}}(y^k).
\end{array}\right.{\!\!\!\!}
\end{equation}
Consequently, the number of iterations to obtain an $\varepsilon$-primal-dual solution $(x^{k+1}, y^k)$ in the sense of Definition~\ref{de:primal_dual_approx_sols} does not exceed:
\begin{equation}\label{eq:complexity_primal_dual}
k_{\max} := \BigO{\sqrt{\nu_f}\ln\left(\frac{\sqrt{\nu_f} t_0}{\varepsilon}\right)},
\end{equation}
where $t_0 \in (0, 1]$ and $\nu_f$ is the barrier parameter of $f$.
\end{theorem}

\begin{proof}
From Step \ref{step:A1_step2} of Algorithm \ref{alg:A1}, we can see that $x^{k+1} := \widetilde{x}^{\ast}_{t_{k+1}}(y^k) \in \intx{\Kc}$.
Moreover, Step \ref{step:A1_step2} also leads to
\begin{equation}\label{eq:tm2_est1}
{\!\!\!\!\!}\begin{array}{ll}
\abss{\nabla{g}(x^{k\!+\!1}) - A^{\top}y^k}^{\ast}_{x^{k\!+\!1}, t_{k\!+\!1}} {\!\!}& \leq \abss{\nabla{g}(x^{k\!+\!1}) - A^{\top}y^k + t_{k+1}\nabla{f}(x^{k+1})}^{\ast}_{x^{k+1}, t_{k+1}}  \vspace{1ex}\\
& + {~} t_{k+1}\abss{\nabla{f}(x^{k+1})}_{x^{k+1},t_{k+1}}^{\ast} \vspace{1ex}\\
&\leq \frac{\delta_kt_{k+1}}{1+\delta_k} + t_{k+1}\abs{\nabla{f}(x^{k+1})}_{x^{k+1},t_{k+1}}^{\ast}.
\end{array}{\!\!\!\!\!}
\end{equation}
Next, for $t \in (0, 1]$, it is obvious that  $\nabla^2{\psi_t}(x;y) =  \tfrac{M_t^2}{4}\big[\nabla^2{g}(x) + t\nabla^2{f}(x)\big] = \nabla^2{f}(x) + \frac{1}{t}\nabla^2{g}(x)$.
Consequently, one has $\nabla^2{\psi_t}(x;y) \succeq \nabla^2{f}(x)$.
Using this fact, we can easily show that
\begin{equation*}
\begin{array}{l}
\abs{\nabla{f}(x^{k+1})}_{x^{k+1},t_{k+1}}^{\ast} \leq \norms{\nabla{f}(x^{k+1})}_{x^{k+1}}^{\ast} \leq \sqrt{\nu_f}. 
\end{array}
\end{equation*}
Combining this inequality and \eqref{eq:tm2_est1}, we obtain the second estimate of \eqref{eq:lm41_res1}.

Now, from \eqref{eq:dual_newton_scheme}, we have
\vspace{-0.75ex}
\begin{equation*} 
-\widetilde{\nabla}^2{d_{t_{k+1}}}(y^k)(\bar{y}^{k+1} - y^k)  - \widetilde{\nabla}{d_{t_{k+1}}}(y^k) \in \partial{h_{t_{k+1}}}(\bar{y}^{k+1}).
\vspace{-0.75ex}
\end{equation*}
Using \eqref{eq:inexact_oracle} and the definition of $h_t$, the last estimate becomes
\vspace{-0.75ex}
\begin{equation*} 
-t_{k+1}\widetilde{\nabla}^2{d_{t_{k+1}}}(y^k)(\bar{y}^{k+1} - y^k)    \in Ax^{k+1} - \partial{\phi^{\ast}}(-\bar{y}^{k+1}).
\vspace{-0.75ex}
\end{equation*}
If we define $r^k :=  y^k - \bar{y}^{k+1}$ and $e^k := t_{k+1}\widetilde{\nabla}^2{d_{t_{k+1}}}(y^k)(\bar{y}^{k+1} - y^k)$, then the last expression leads to
\vspace{-0.75ex}
\begin{equation*} 
-e^k   \in Ax^{k+1} - \partial{\phi^{\ast}}(-y^k + r^k)~~\Leftrightarrow~~r^k \in y^k + \partial{\phi}(Ax^{k+1}+ e^k).
\vspace{-0.75ex}
\end{equation*}
It is obvious to show that
\begin{equation*}
\tnorms{e^k}_{y^k, t_k+1}^{\ast} = t_{k+1}\tnorms{y^k - \bar{y}^{k+1}}_{y^k, t_k+1} = t_{k+1}\lambda_{t_{k+1}}(y^k),
\end{equation*}
which is the first statement in the last line of \eqref{eq:lm41_res1}.

Now, from \eqref{eq:inexact_oracle_k} and $t_{k+1} \in (0,1]$, we have $\widetilde{\nabla}^2d_{t_{k+1}}(y^k) = \frac{1}{t_{k+1}^2}A\big(\nabla^2f(x^{k+1}) + \frac{1}{t_{k+1}}\nabla^2g(x^{k+1})\big)^{-1}A^{\top}$. 
This implies that
\begin{equation*}
\begin{array}{ll}
\lambda_{t_{k+1}}(y^k)^2 &= (\bar{y}^{k+1} - y^k)^{\top}\widetilde{\nabla}^2d_{t_{k+1}}(y^k)(\bar{y}^{k+1} - y^k) \vspace{1ex}\\
&= \frac{1}{t_{k+1}^2}(r^k)^{\top}A\big(\nabla^2f(x^{k+1}) + \frac{1}{t_{k+1}}\nabla^2g(x^{k+1})\big)^{-1}A^{\top}r^k \vspace{1ex}\\
&= \frac{1}{t_{k+1}^2}\big(\abss{A^{\top}r^k}_{x^{k+1}, t_{k+1}}^{\ast}\big)^2.
\end{array}
\end{equation*}
Therefore, we have $\abss{A^{\top}r^k}_{x^{k+1}, t_{k+1}}^{\ast} = t_{k+1}\lambda_{t_{k+1}}(y^k)$, which proves the second statement in the last line of \eqref{eq:lm41_res1}.

From \eqref{eq:lm41_res1}, to obtain an $\varepsilon$-primal-dual solution $(x^{k+1}, y^k)$ in the sense of Definition~\ref{de:primal_dual_approx_sols}, we need to set $\left(\sqrt{\nu_f} + \frac{\delta_k}{1+\delta_k}\right)t_{k+1} \leq \varepsilon$ and $t_{k+1}\lambda_{t_{k+1}}(y^k) \leq \varepsilon$.
Since $\lambda_{t_{k+1}}(y^k) \leq 2.1\beta$ (see the proof in Theorem \ref{thm:one_step_analysis}) and $\delta_k \leq \frac{\beta}{100}$, we can set $t_{k+1} \leq \hat{\varepsilon}$ such that
\begin{equation*}
\varepsilon \geq  \hat{\varepsilon}(\sqrt{\nu_f} + 1) \geq \hat{\varepsilon} \max\set{\sqrt{\nu_f} + \frac{0.01\beta}{1+0.01\beta}, 2.1\beta},
\end{equation*}
i.e., $\hat{\varepsilon} \leq  \frac{\varepsilon}{(1+\sqrt{\nu_f})}$. Combining this expression and \eqref{eq:iter_complexity_bound}, we can show that the number of iterations to obtain an $\epsilon$-primal-dual solution does not exceed $\BigO{\sqrt{\nu_f}\ln\left(\frac{\sqrt{\nu_f}t_0}{\varepsilon}\right)}$, which is exactly \eqref{eq:complexity_primal_dual}.
\Eproof
\end{proof}

\beforepara
\paragraph{\textbf{Discussion:}}
Theorem~\ref{th:key_property7} estimates the maximum iterations $k_{\max}$ to obtain an $\varepsilon$-primal-dual solution $(x^{k+1}, y^k)$ of \eqref{eq:constr_cvx} and \eqref{eq:dual_prob}.
It shows that such a number of iterations remains the same as in standard path-following methods \cite{Nesterov2004} up to a constant factor.
Although the norms in \eqref{eq:lm41_res1} are local norms, but this is the standard metric used in general interior-point methods \cite{Nesterov2004,Nesterov1994}.

\beforesec
\subsection{\bf Finding an initial point in Algorithm~\ref{alg:A1}}\label{subsec:initial_point}
\aftersec
We need to find $(x^0, y^0)$ such that the condition \eqref{eq:initial_point_cond} holds.
As in standard interior-point methods, we need to perform a damped proximal-Newton method.
Such a method can be found in, e.g. \cite{TranDinh2013e,Tran-Dinh2013a}, but since we use inexact oracles, we need to customize this method in our context.
More specifically, we describe this routine in Algorithm~\ref{alg:A2}.

\begin{algorithm}[ht!]\caption{(\textit{Find an initial point $x^0,y^0$})}\label{alg:A2}
\begin{algorithmic}[1]
\normalsize
\STATE\textbf{Initialization.}\label{step:A2_step0}~Choose an initial point $\hat{y}^0 \in\R^n$ and fix a value $t_0 \in (0, 1]$.
\STATE\textbf{Main iteration.}~\textrm{For $j = 0$ to $j_{\max}$, perform}
\vspace{0.75ex}
\STATE\hspace{0.2cm}\label{step:A2_step2} Solve approximately \eqref{eq:opt_cond_subprob} at $y = \hat{y}^j$ up to an accuracy $\delta_j \in (0, \frac{\beta}{100}]$ to get $\hat{x}^j := \tilde{x}_{t_0}^{\ast}(\hat{y}^j)$, i.e.:
\begin{equation*}
\abss{\nabla\psi_{t_0}(\hat{x}^j; \hat{y}^j)}_{\hat{x}^j, t_0}^{\ast} \leq \tfrac{\delta_j}{1+\delta_j}.
\vspace{-1.5ex}
\end{equation*}
\STATE\hspace{0.2cm}\label{step:A2_step3}(\textbf{Inexact oracles}): Evaluate inexact gradient and Hessian of $d_{t_0}$ as
\begin{equation}\label{eq:inexact_oracle_j}
\left\{\begin{array}{ll}
\widetilde{\nabla}{d_{t_0}}(\hat{y}^j) &:= \frac{M_{t_0}^2}{4}A\hat{x}^j ,\vspace{1ex}\\
\widetilde{\nabla}^2{d_{t_0}}(\hat{y}^j) &:= \frac{M_{t_0}^4}{16}A\nabla^2{\psi_{t_0}}(\hat{x}^j)^{-1}A^{\top}.
\end{array}\right.
\vspace{-1ex}
\end{equation}
\STATE\hspace{0.2cm}\label{step:A2_step4}(\textbf{Inexact damped-step proximal-Newton step}):
Compute $\hat{s}^j$ up to an accuracy $\epsilon_j \in (0, \frac{\beta}{100}]$  and update $\hat{y}^{j}$, i.e.:
\begin{equation*}
\left\{\begin{array}{ll}
\hat{s}^j &:\approx s^j:=\text{prox}_{h_{t_0}}^{\widetilde{\nabla}^2d_{t_0}(\hat{y}^j)}\left( \hat{y}^j - \widetilde{\nabla}^2d_{t_0}(\hat{y}^j)^{-1}\widetilde{\nabla}d_{t_0}(\hat{y}^j) \right) \vspace{1ex}\\
\hat{y}^{j+1} &:= (1-\alpha_j)\hat{y}^j + \alpha_j\hat{s}^j,
\end{array}\right.
\end{equation*}
where $\alpha_j := \frac{(\hat{\lambda}_j - \epsilon_j - \delta_j)(1-\delta_j)^2}{\big((1-\delta_j)(\hat{\lambda}_j - \epsilon_j - \delta_j) + 1\big)\hat{\lambda}_j} \in (0, 1)$ and $\hat{\lambda}_j := \tnorms{\hat{s}^j - \hat{y}^j}_{\hat{y}^j,t_0}$. 
\STATE\hspace{0cm}\textbf{End.}
\end{algorithmic}
\end{algorithm}

We terminate Algorithm~\ref{alg:A2} if we find $x^0 := \hat{x}^{j_{\max}}$ and $y^0 := \hat{y}^{j_{\max}}$ such that \eqref{eq:initial_point_cond} holds. 
Since the constraint of $x^0$ in \eqref{eq:initial_point_cond} is always satisfied from Step~\ref{step:A2_step2} of Algorithm~\ref{alg:A2}, we only need to guarantee that $\lambda_{t_0}(y^0) \leq \beta$.

The following theorem estimates the number of iterations to obtain $(x^0, y^0)$ satisfying  \eqref{eq:initial_point_cond}.

\begin{theorem}\label{th:initial_point}
Let us define $\hat{\lambda}_j := \tnorms{\hat{s}^j - \hat{y}^j}_{\hat{y}^j,t_0}$ and $\lambda_j := \tnorms{s^j - \hat{y}^j}_{\hat{y}^j,t_0}$. Let $\sets{(\hat{x}^j, \hat{y}^j)}$ be the sequence generated by Algorithm~\ref{alg:A2}, where we choose $\delta_j, \epsilon_j \in \big(0, \frac{\beta}{100}\big]$ and the step-size 
\begin{equation}\label{eq:step-size}
\alpha_j := \frac{(\hat{\lambda}_j - \epsilon_j - \delta_j)(1-\delta_j)^2}{\big[1 + (1-\delta_j)(\hat{\lambda}_j - \epsilon_j - \delta_j)\big]\hat{\lambda}_j} \in (0, 1). 
\end{equation}
Then, after at most finite number of iterations $j_{\max}$ as
\begin{equation}\label{eq:initial_point}
j_{\max} := \left\lfloor \frac{D_{t_0}(\hat{y}^0) - D_{t_0}(y^{\ast}_{t_0})}{\omega\left(0.97\beta(1 - 10^{-2}\beta)\right)} \right\rfloor + 1,
\end{equation}
we obtain $y^0 := \hat{y}^{j_{\max}}$ and $x^0 := \hat{x}^{j_{\max}}$ such that $\lambda_{t_0}(y^0) \leq \beta$ and \eqref{eq:initial_point_cond} holds, where $y^{\ast}_{t_0}$ is the optimal solution of \eqref{eq:smoothed_dual_prob} at $t := t_0$.
\end{theorem}

\begin{proof}
Note that at each iteration $j$ of Algorithm~\ref{alg:A2}, we always have $\lambda_j > \beta$.
By the triangle inequality and the choice of $\epsilon_j$, we can easily show that
\begin{equation*}
\hat{\lambda}_j  \geq \tnorms{s^j - \hat{y}^j}_{\hat{y}^j,t_0} - \tnorms{s^j - \hat{s}^j}_{\hat{y}^j,t_0} \geq \lambda_j - \epsilon_j > (1-10^{-2})\beta.
\end{equation*}
In addition, from Lemma~\ref{le:initial_point} in  Appendix~\ref{apdx:th:initial_point}, we have
\begin{equation*}
D_{t_0}(\hat{y}^{j+1})  \leq D_{t_0}(\hat{y}^{j}) -  \omega\left((\hat{\lambda}_j - \epsilon_j - \delta_j)(1-\delta_j)\right),
\end{equation*}
where $\omega(\tau) := \tau - \ln(1+\tau) \geq 0$.
Using $\epsilon_j \leq 10^{-2}\beta$, $\delta_j \leq 10^{-2}\beta$, and $\hat{\lambda}_j \geq (1-10^{-2})\beta$ in the above inequality, we get
\begin{equation*}
D_{t_0}(\hat{y}^{j+1})  \leq D_{t_0}(\hat{y}^{j}) - \omega\left(0.97\beta(1 - 10^{-2}\beta)\right).
\end{equation*}
Summing up this inequality from $j=0$ to $j = j_{\max}$, we obtain
\begin{equation*}
j_{\max} \omega(0.97\beta(1 - 10^{-2}\beta)) \leq D_{t_0}(\hat{y}^{0}) - D_{t_0}(\hat{y}^{j_{\max}}) \leq D_{t_0}(\hat{y}^{0}) - D_{t_0}(y^{\star}_{t_0}),
\end{equation*}
which implies $j_{\max} \leq \frac{D_{t_0}(\hat{y}^{0}) - D_{t_0}(y^{\star}_{t_0})}{\omega(0.97\beta(1 - 10^{-2}\beta))}$.
Consequently, we obtain \eqref{eq:initial_point}.
\Eproof
\end{proof}

\beforepara
\paragraph{\textbf{Discussion:}}
Theorem~\ref{th:initial_point} shows that the number of iterations to obtain a starting point $(x^0, y^0)$ is finite even with inexact oracles and inexact proximal-Newton methods.
However, the convergence rate of Algorithm~\ref{alg:A2} is sublinear in $j$.
If $t_0$ is large (i.e., close to $1$), Algorithm~\ref{alg:A2} often requires a small number of iterations.
Another possibility is to apply a path-following procedure as in \cite{TranDinh2015f} to obtain a new variant with linear convergence rate.
Note that the per-iteration complexity of Algorithm~\ref{alg:A2} is essentially the same as in Algorithm~\ref{alg:A1} since the computation of $\hat{\lambda}_j$ is neglectable.  
In particular, if we choose $\epsilon_j = \delta_j = 0$, the steps size $\alpha_j$ will become the standard damped Newton step-size $\frac{1}{1+\lambda_j}$ in the theory of self-concordant function \cite{Nesterov1994}.

\beforesec
\section{Numerical Experiments}\label{sec:numerical_experiment}
\aftersec
We provide two numerical examples to illustrate our algorithm and compare it with some existing methods.
We choose SDPT3 \cite{Toh2010} as a common used conic solver, and Chambolle-Pock's (CP) primal-dual method \cite{Chambolle2011} as one of the most powerful first-order methods that can handle our problem.
The first example is the well-known network utility maximization (NUM) problem, and the second one is the spectrum management problem for multi-user DSL networks studied in \cite{Tsiaflakis2008}.
Our method and the CP method are implemented in Matlab 2018b, running on a Linux server with 3.4GHz Intel Xeon E5 and 16Gb memory.

\beforesubsec
\subsection{\bf Implementation remarks}\label{subsec:impl_details}
\aftersubsec
We discuss how we implement two main steps of Algorithm~\ref{alg:A1}  as follows.
First, we need to solve the slave problem at Step~\ref{step:A1_step2} up to a given accuracy $\delta_k$ such that $\delta_k \leq 10^{-2}\beta$.
Solving this problem is equivalent to solving the nonlinear equation $\nabla\psi_{t_{k+1}}(x; y^k) = 0$ in $x$.
Since $\psi_{t_{k+1}}(\cdot; y^k)$ is standard self-concordant, we can apply a damped-step Newton method to solve it.
Combining this method and a warm-start strategy, we can solve this equation efficiently.
Second, if $\phi =\delta_{\set{b}}$ in \eqref{eq:constr_cvx} for a given $b\in\R^n$, then the master problem at Step~\ref{step:A1_step4} reduces to a positive definite linear system $\widetilde{\nabla}^2d_{t_{k+1}}(y^k)(y - y^k) = -\widetilde{\nabla}d_{t_{k+1}}(y^k) + 0.25M_{t_{k+1}}^2b$, which can be efficiently solved by, e.g., preconditioned conjugate gradient methods.
However, since $\phi$ usually does not have such a simple form, we need to apply iterative methods such as accelerated proximal gradient method \cite{Beck2009,Nesterov2004} to solve this problem which has a linear convergence rate.
Note that we can also apply a semi-smooth Newton-type methods as in \cite{yang2015sdpnalp} to solve this problem efficiently.
In our numerical test, we use FISTA which seems working well.

\beforesubsec
\subsection{\bf Network Utility Maximization}\label{subsec:NUM_exam}
\aftersubsec
Consider a network consisting of a finite set $\Sc$ of  $N$ nodes and a finite set $\mathcal{E}$ of undirected capacitated edges. 
Let $x_{ij}$ denote the rate of sending data from node $i$ to  node $j$.
We assume that such a flow $f_{ij}$ from node $i$ to node $j$ is fixed and unique (we usually choose $f_{ij}$ to be the shortest path from $i$ to $j$).

Assume that each node $i$ is associated with a utility function $u_i(x_i) := \log\big(d_i^{\top}x_i + \mu_i\big)$, where $x_i := (x_{i1}, \cdots, x_{iN})^{\top}$, $d_i := (d_{i1},\cdots, d_{iN})^{\top}$ and $\mu_i$ is a scalar. 
Since we ignore self-links from node $i$ to itself,  we set $d_{ii} = 0$ and $f_{ii} = \emptyset$. 
We further assume that the rate $x_{ij}$ is constrained to lie in a given interval $[0, M]$, where the scalar $M$ denotes the maximum capacity of flows.

Under this setting, we formulate the problem of interest into the following constrained convex optimization problem called NUM:
\vspace{-0.75ex}
\begin{equation}\label{eq:NUM_exam}
\left\{\begin{array}{cl}
\displaystyle\max_{x} &\Big\{ \displaystyle\sum_{i\in\Sc} \ln(d_i^{\top}x_i + \mu_i) - \frac{\rho}{2}\norms{x_i - r_i}^2  \Big\} \vspace{1ex}\\ 
 \textrm { s.t. } &  L_e \leq \displaystyle\sum_{ e \in f_{ij}} x_{ij} \leq U_e,~~\forall e \in \mathcal{E}, \vspace{1ex}\\ 
 & 0 \leq x_{ij} \leq M,~~~ \forall i,j \in \Sc. 
\end{array}\right.
\vspace{-0.75ex}
\end{equation}
Here, $L_e$ and $U_e$ are the lower bound and upper bound  capacity of each edge, respectively, 
$r_{ij}$ is the initial designed rate from node $i$ to node $j$ and we do not want to have the rate $x_{ij}$ to be far away from our target $r_{ij}$, and $\rho$ is the corresponding penalty parameter to control the distance from $x_{ij}$ to $r_{ij}$. 
By defining $g(x) := -\sum_{i\in\Sc} \ln(d_i^{\top}x_i + \mu_i) + \frac{\rho}{2}\norms{x_i - r_i}^2$, $Ax = \sum_{e\in f_{ij}}x_{ij}$, $\phi(\cdot) := \delta_{[L_e, U_e]}(\cdot)$, and $\Kc := [0, M]$, we can reformulate \eqref{eq:NUM_exam} into \eqref{eq:constr_cvx}.
Clearly, this problem satisfies Assumptions~\ref{as:A0} and \ref{as:A1}.

We implement Algorithm~\ref{alg:A1} using Algorithm~\ref{alg:A2} to find an initial point using $t_0 := 0.25$.
We also implement the Chambolle-Pock method in \cite{Chambolle2011} and use SDPT3 to solve \eqref{eq:NUM_exam} as our competitors.
Note that SDPT3 can directly handle $\log$-terms in $g$ compared to other interior-point solvers such as SeDuMi, SDPA, or Mosek.
To avoid solving subproblems in the Chambolle-Pock method, we reformulate \eqref{eq:NUM_exam} by introducing auxiliary variables $z_i := d_i^{\top}x_i + \mu_i$ for $i\in\Sc$.
Since the Chambolle-Pock method has two step-sizes $\tau$ and $\sigma$, we tune $\tau$ for each run and let $\sigma := 0.99/(\tau\norms{K}^2)$, where $K$ is the linear operator obtained from reformulating \eqref{eq:NUM_exam} into a composite form. The best values of $\tau$ we found are between $10^{-6}$ and $10^{-7}$ depending on problem.

All algorithms are terminated  when both infeasibility and relative duality gap reach $10^{-7}$ accuracy or the maximum number of iterations $k_{\max} := 20,000$ is exceeded.
In the first case, we certify that the problem is ``solved'', while in the second case, we mark it by ``*''.
If problem is too big to solve by our computer, we also mark it by ``*''.

We use the ``tech-router-rf'' dataset from \href{http://networkrepository.com/tech-routers-rf.php}{http://networkrepository.com/tech-routers-rf.php} from \cite{rocketfuel}, where we have approximately $2000$ nodes and $6000$ edges. In this network, each node is either a router or a computer IP. Each computer IP has to go through one or multiple routers to send data to another computer IP. 
For larger networks, we use the ``tech-pgp'' dataset from \href{http://networkrepository.com/tech-pgp.php}{http://networkrepository.com/tech-pgp.php} from \cite{boguna2004models}, which is a social network with approximately $11000$ nodes and $24000$ edges.
Given a network structure, we generate the input data as follows.
The initial designed rate $r_i$ are generated from a uniform distribution $\Uc(0, 1)$ between $0$ and $1$.
The upper and lower bounds of capacity are generated as $L_e := (1 - \Uc(0, 0.5))\bar{b}$ and $U_e := (1 + \Uc(0, 0.5))\bar{b}$, where $\bar{b} := \sum_{i \in \Sc} A_{i}r_{i}$. 
The maximum limit of rate $M$ is $1$ and the penalty paramter $\rho$ is chosen to be $0.01$. 
Both $d_i$ and $\mu_i$ are generated randomly using $\Uc(0, 1)$.
To have different problem instances, we use different sub-networks of the original one.

We run three algorithms on $10$ problems instances of different sizes.
The results are reported in Table \ref{table:result_NUM}, where $n$ is the number of linear inequality constraints, $p$ is the number of variables in \eqref{eq:NUM_exam}, \texttt{IPLD} is Algorithm~\ref{alg:A1}, and CP is the Chambolle-Pock method in \cite{Chambolle2011}.

\begin{table}[!htbp]
\vspace{-2ex}
\newcommand{\cell}[1]{{\!\!}#1{\!\!}}
\caption{Numerical results of three solvers on 10 problem instances of \eqref{eq:NUM_exam}.}
\label{table:result_NUM}
\resizebox{\textwidth}{!}{
\begin{tabular}{| r  r | rrr | rrr | rrr |}
\hline
\multicolumn{2}{|c}{Problem size}  &  \multicolumn{3}{|c}{CPU time [s]}  &     \multicolumn{3}{|c}{Feasibility violation}  &  \multicolumn{3}{|c|}{Objective value $f^{\star}$}          \\ \hline
\cell{$n~~$}   & \cell{$p~~~~$}       & \cell{\texttt{IPLD}} & \cell{CP}     & \cell{SDPT3}  & \cell{\texttt{IPLD}}    & \cell{CP}        & \cell{SDPT3}     & \cell{\texttt{IPLD}}       &  \cell{CP}         & \cell{SDPT3}     \\ \hline
\cell{96}  & \cell{17,686}   & \cell{0.70}   & \cell{3.80}   & \cell{4.24}   & \cell{4.747e-09} & \cell{9.986e-08} & \cell{0.000e+00} & \cell{-60.1375}   & \cell{-60.1375}   & \cell{-60.1375}  \\ \hline
\cell{188} & \cell{29,502}   & \cell{1.31}   & \cell{4.44}   & \cell{9.59}  & \cell{5.126e-09} & \cell{9.974e-08} & \cell{0.000e+00} & \cell{-116.8216}   & \cell{-116.8216}   & \cell{-116.8216}  \\ \hline
\cell{239} & \cell{38,050}   & \cell{1.85}   & \cell{6.64}   & \cell{11.75}  & \cell{7.227e-10} & \cell{9.983e-08} & \cell{0.000e+00} & \cell{-171.4066}  & \cell{-171.4066}  & \cell{-171.4066} \\ \hline
\cell{306} & \cell{53,048}   & \cell{2.43}   & \cell{9.45}   & \cell{227.32}  & \cell{2.266e-10} & \cell{9.995e-08} & \cell{0.000e+00} & \cell{-228.6001}  & \cell{-228.6001}  & \cell{-228.6001} \\ \hline
\cell{242} & \cell{72,016}   & \cell{2.78}   & \cell{10.25}   & \cell{809.57} & \cell{9.055e-09} & \cell{9.982e-08} & \cell{0.000e+00} & \cell{-288.6970}  & \cell{-288.6970}  & \cell{-272.1732} \\ \hline
\cell{324} & \cell{125,848}  & \cell{5.61}   & \cell{26.98}  & *  & \cell{1.107e-08} & \cell{9.987e-08} & * & \cell{-569.2405}  & \cell{-569.2405}  & \cell{*}   \\ \hline
\cell{658} & \cell{243,936}  & \cell{19.37}  & \cell{78.58}  & *  & \cell{5.478e-09} & \cell{9.987e-08} & * & \cell{-1133.8747}  & \cell{-1133.8747}  & \cell{*}   \\ \hline
\cell{833} & \cell{432,218}  & \cell{47.86} & \cell{203.95} & *     & \cell{8.007e-08} & \cell{9.999e-08} & * & \cell{-2124.3265} & \cell{-2124.3265} & \cell{*}   \\ \hline
\cell{1,383} &\cell{1,194,500}  & \cell{206.88} & \cell{571.17} & *     & \cell{9.473e-08} & \cell{9.983e-08} & * & \cell{-3236.1724} & \cell{-3236.1724} & \cell{*}   \\ \hline
\cell{1,619} & \cell{2,389,000}  & \cell{556.34} & \cell{1297.86} &  *     & \cell{4.422e-09} & \cell{9.994e-08} & * & \cell{-6474.3812} & \cell{-6474.3812} & \cell{*}   \\ \hline
\end{tabular}}
\vspace{-2ex}
\end{table}

From Table \ref{table:result_NUM}, we observe the following facts:
\begin{itemize}
\vspace{-1.25ex}
\item \texttt{IPLD} can solve large-scale problems with huge variables and moderate number of couple linear inequality constraints relatively fast and accurate.
\texttt{IPLD} outperforms SDPT3 and CP in a majority of problems in terms of CPU time and achieves the same accuracy in the objective value and constraint violation.

\item It is not surprising that CP can also achieve high accuracy but requires very large number of iterations.
The CP algorithm requires from $6500$ to $15200$ iterations to achieve our specified accuracy depending on problem.

\item SDPT3 is quickly prohibited to handle larger instances due to the increase of  variables and constraints when transforming it into a conic and $\log$ form.
Therefore, the problem cannot be fit into our computer memory.
\vspace{-1.25ex}
\end{itemize}
In summary, we believe that our method, \texttt{IPLD}, can potentially solve large-scale convex problems of the form \eqref{eq:constr_cvx} as long as they satisfy Assumptions~\ref{as:A0} and \ref{as:A1}.
It can often achieve high accuracy within reasonably computational effort and can be easily parallelized. 
While primal-dual first-order methods require to tune the step-size to obtain good performance, our method is relatively robust to inexact oracles and inexact Newton-type methods as well as the choice of parameter $t_0 \in (0, 1]$.

\beforesubsec
\subsection{\bf Spectrum management of multi-user DSL networks}\label{subsec:network_exam}
\aftersubsec
We consider the spectrum management problem of multi-user DSL networks studied in \cite{Tsiaflakis2008}, which can be cast into the following constrained problem:
\vspace{-0.75ex}
\begin{equation}\label{eq:DSL_exam1}
\left\{\begin{array}{cl}
\displaystyle\min_{x\in\R^m} & \Big\{ g(x) := -\sum_{i=1}^{M}\big[ a_i^{\top}x_i - c_i^{\top}\ln (H_ix_i + g_i)\big] \Big\}  \vspace{1ex}\\ 
\mathrm{s.t.}~ &  \sum_{i=1}^Mx_i \leq b, \vspace{1ex}\\ 
& 0 \leq x_i \leq L, ~i=1,\cdots, M.
\end{array}\right.
\vspace{-0.75ex}
\end{equation}
where $x_i \in \R^{m}$, ~$a_i \in \R^{m}$, $c_i \in \R^{m}_{+}$, $b \in \R^m$, $L \in \R^m_{++}$, $g_i \in \R^{m}$,  and $H_i \in \R^{m\times m}$.
Here, $m$ is the number of users, and $M$ is the number of channels.
For the detail explanation of this model, we refer the reader to \cite{Tsiaflakis2008}.
Clearly, \eqref{eq:DSL_exam1} can be cast into \eqref{eq:constr_cvx}, where $g$ is self-concordant, $Ax =  \sum_{i=1}^Mx_i$, $\phi := \delta_{(-\infty, b]}$ the indicator of $(-\infty, b]$, and $\Kc := [0, L]^M$.

Our goal in this example is to verify the performance of Algorithm~\ref{alg:A1} using different accuracy levels both for inexact oracles and inexact proximal-Newton method.
For this purpose, we use two real datasets  to test our algorithm.
More precisely,  we first fix the tolerance $\delta_k$ of the inexact oracles at $10^{-5}$ and change the tolerance $\epsilon_k$ of the inexact proximal-Newton method from $10^{-2}$ to $10^{-11}$. Then, we fix the tolerance $\epsilon_k$ at $10^{-5}$ in the inexact proximal-Newton scheme  and vary  $\delta_k$ in the inexact oracles between $10^{-2}$ and $10^{-8}$. 
In all these cases, we terminate our algorithm whenever the feasibility violation is below $10^{-5}$  and the relative gap is below $10^{-6}$.

In the first test, we use  a $7$-user asymmetric ADSL downstream dataset, where $m = 7$ and $M = 224$. 
Figure~\ref{fig:DSL7} shows how the number of iterations and the normalized CPU time depend on the tolerances, where the normalized CPU time is computed by $(T - T_{\min})/(T_{\max} - T_{\min})$ with the time $T$.

\begin{figure}[hpt!]
\begin{center}
\includegraphics[width=1\textwidth]{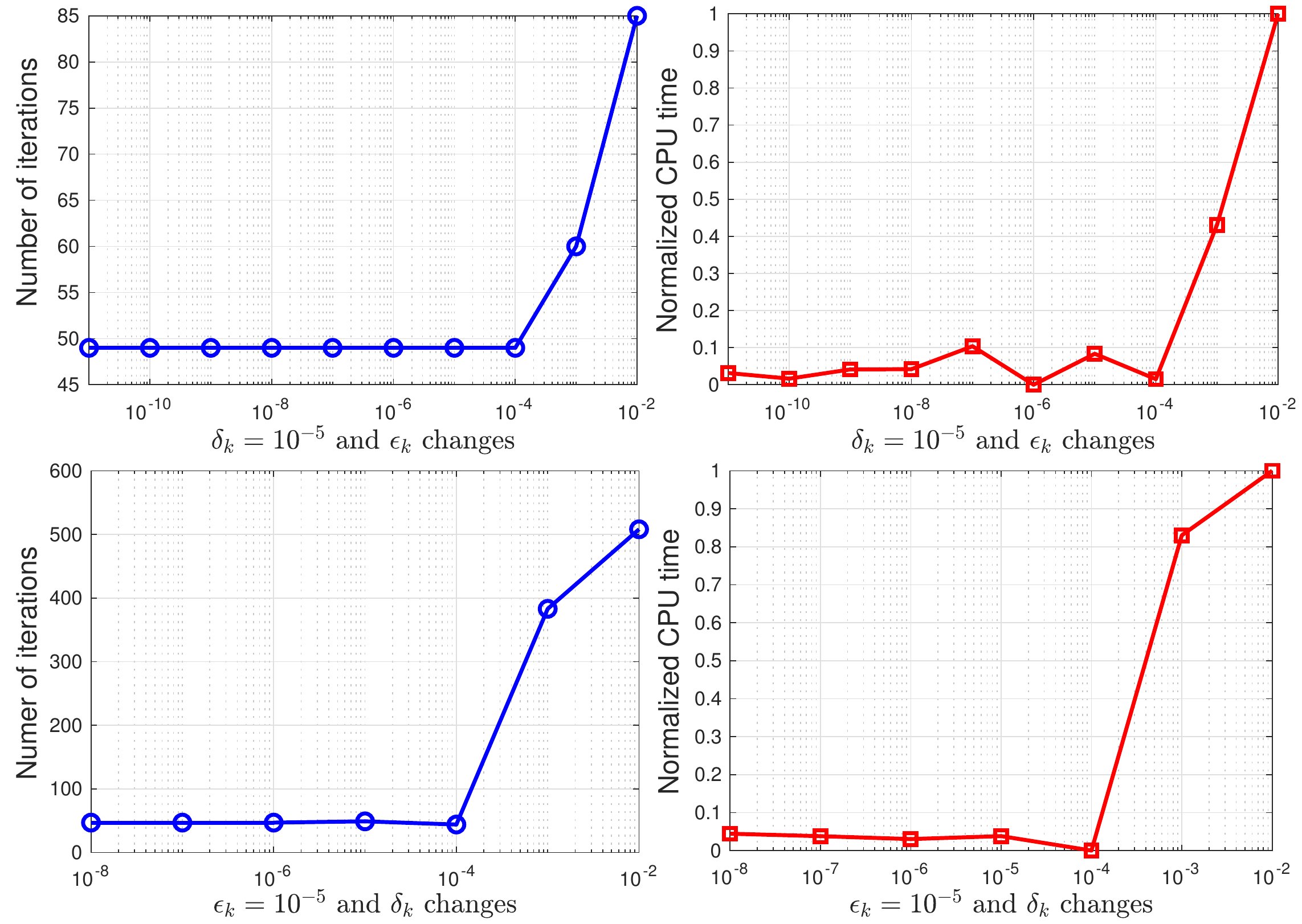}
\vspace{-1.5ex}
\caption{The number of iterations and normalized CPU time of Algorithm~\ref{alg:A1} on the $7$-users dataset.
The first row shows the number of iterations and normalized CPU time when $\delta_k$ is fixed at $10^{-5}$ and $\epsilon_k$ changes from $10^{-12}$ to $10^{-2}$, while the second row is for the case $\epsilon_k = 10^{-5}$ and $\delta_k$ changes from $10^{-8}$ to $10^{-2}$.}
\label{fig:DSL7}
\end{center}  
\vspace{-6ex}
\end{figure}

We can see from the top row of Figure~\ref{fig:DSL7} that with $\delta_k = 10^{-5}$ fixed and $\epsilon_k \leq 10^{-4}$, the number of iterations is almost stable and the computational time does not decrease significantly.
This suggests that the accuracy $\epsilon_k = 10^{-4}$ is sufficiently for computing proximal-Newton direction in the dual problem.
If $\epsilon_k > 10^{-4}$, then the number of iterations and CPU time increase significantly.
Similarly, if we fix $\epsilon_k = 10^{-5}$ and increase $\delta_k$ from $10^{-8}$ to $10^{-2}$, then we can observe from the bottom row of  Figure~\ref{fig:DSL7} that $\delta_k \leq 10^{-4}$ is sufficient to accommodate the inexact oracles. 

To confirm our above statement, we again test our algorithm with the second dataset, 12-user VDSL upstream dataset,  where $n = 12$ and $M = 1147$.
Figure~\ref{fig:DSL12} provides the number of iterations and normalized CPU time by rescaling it between $[0, 1]$ as in Figure~\ref{fig:DSL7}.
We again observe very similar behavior in both situations, but since the problem is relatively larger than that of the first dataset, the computational time increases significantly when we decrease the accuracy $\delta_k$ of the inexact oracles.  

\begin{figure}[!htp]
\begin{center}
\includegraphics[width=1\textwidth]{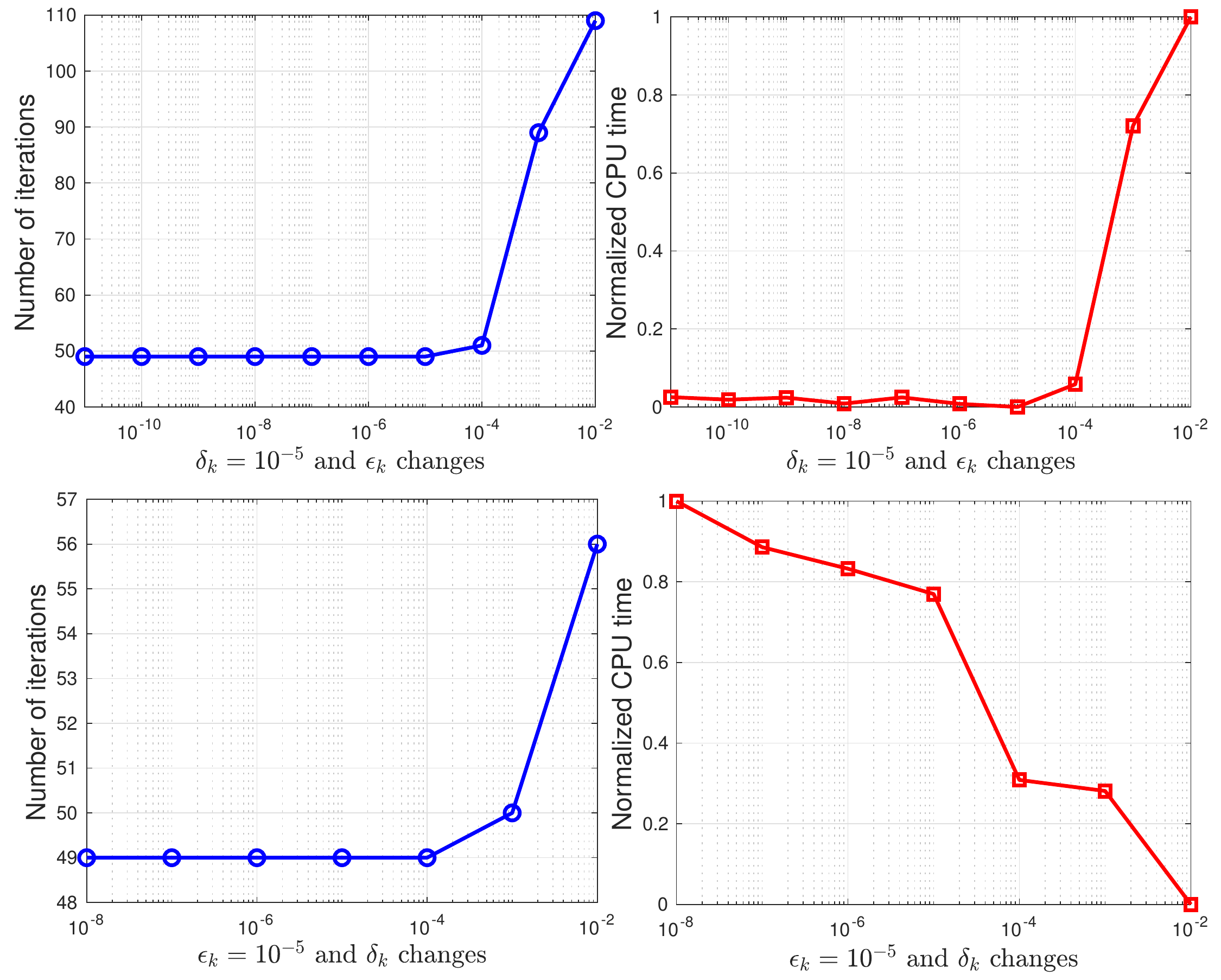}
\vspace{-1.5ex}
\caption{The number of iterations and normalized CPU time of Algorithm~\ref{alg:A1} on the $12$-users dataset.
The first row shows the number of iterations and normalized CPU time when $\delta_k$ is fixed at $10^{-5}$ and $\epsilon_k$ changes from $10^{-12}$ to $10^{-2}$, while the second row is for the case $\epsilon_k = 10^{-5}$ and $\delta_k$ changes from $10^{-8}$ to $10^{-2}$.}
\label{fig:DSL12}
\end{center}  
\vspace{-6ex}
\end{figure}

\beforepara
\begin{acknowledgements}
This work was partly supported by the National Science Foundation (NSF), awarded number: DMS-1619884. 
\end{acknowledgements}

\appendix  
\beforesec
\section{Appendix:  The proof of technical results in the main text}
\aftersec
We provide all the missing proofs in the main text.

\beforesubsec
\subsection{\bf The proof of Proposition~\ref{pro:inexact_oracle}: Properties of inexact oracles.}\label{apdx:pro:inexact_oracle}
\aftersubsec
Since $x^{\ast}_t(y)$ is the exact solution of \eqref{eq:opt_cond_subprob}, we have 
\begin{equation*}
\nabla{\psi_t}(x^{\ast}_t(y); y) \equiv \tfrac{M_t^2}{4}\big[\nabla{g}(x^{\ast}_t(y)) + t\nabla{f}(x^{\ast}_t(y)) - A^{\top}y\big] = 0.
\end{equation*}
Therefore, using the standard self-concordance of $\psi_t$, we can show that
\begin{equation*}
\begin{array}{ll}
\iprods{\nabla{\psi_t}(\widetilde{x}^{\ast}_t(y); y), \widetilde{x}^{\ast}_t(y)-x^{\ast}_t(y)} &= \iprods{ \nabla{\psi_t}(\widetilde{x}^{\ast}_t(y); y) - \nabla{\psi_t}(x^{\ast}_t(y); y), \widetilde{x}^{\ast}_t(y) - x^{\ast}_t(y)}\vspace{1ex}\\
&\geq \frac{\abss{\widetilde{x}^{\ast}_t(y) - x^{\ast}_t(y)}_{\widetilde{x}^{\ast}_t(y), t}^2}{1 {~} + {~} \abss{\widetilde{x}^{\ast}_t(y) - x^{\ast}_t(y)}_{\widetilde{x}^{\ast}_t(y), t}}.
\end{array}
\end{equation*}
By the Cauchy-Schwarz inequality, we have
\begin{equation*}
\iprods{\nabla{\psi_t}(\widetilde{x}^{\ast}_t(y); y), \widetilde{x}^{\ast}_t(y)-x^{\ast}_t(y)} \leq \abss{\nabla{\psi_t}(\widetilde{x}^{\ast}_t(y); y)}_{\widetilde{x}^{\ast}_t(y), t}^{\ast} \abss{\widetilde{x}^{\ast}_t(y)-x^{\ast}_t(y)}_{\widetilde{x}^{\ast}_t(y), t}.
\end{equation*}
Combining the last two inequalities, we eventually get
\begin{equation*}
\frac{\abss{\widetilde{x}^{\ast}_t(y) - x^{\ast}_t(y)}_{\widetilde{x}^{\ast}_t(y), t}}{1 + \abss{\widetilde{x}^{\ast}_t(y) - x^{\ast}_t(y)}_{\widetilde{x}^{\ast}_t(y), t}} \leq \abss{\nabla{\psi_t}(\widetilde{x}^{\ast}_t(y); y)}_{\widetilde{x}^{\ast}_t(y), t}^{\ast} \leq \frac{\delta}{1 + \delta}.
\end{equation*}
This implies  that $\abss{\widetilde{x}^{\ast}_t(y)-x^{\ast}_t(y)}_{\widetilde{x}^{\ast}_t(y), t} \leq \delta$.

Next, using \eqref{eq:dt_and_derivatives} and \eqref{eq:inexact_oracle}, we have $d_t(y) - \widetilde{d}_t(y) = \psi_t(\widetilde{x}^{\ast}_t(y); y) - \psi_t(x^{\ast}_t(y); y)$.
Therefore, applying \cite[Theorems 4.1.7 and 4.1.8]{Nesterov2004} respectively, we obtain the first estimate of \eqref{eq:inexact_oracle_properties}.

Note that since $\nabla^2{d_t}(y) = \frac{M_t^4}{16}A\nabla^2{\psi_t}(x^{\ast}_t(y); y)^{-1}A^{\top}$ and $\widetilde{\nabla}^2{d_t}(y) = \frac{M_t^4}{16}A\nabla^2{\psi_t}(\widetilde{x}^{\ast}_t(y))^{-1}A^{\top}$, using \cite[Theoryem 4.1.6]{Nesterov2004}, we obtain the second estimate of \eqref{eq:inexact_oracle_properties}.

Finally, since $\nabla{d}_t(y) - \widetilde{\nabla}{d_t}(y) = \tfrac{M_t^2}{4}A(x^{\ast}_t(y) - \widetilde{x}^{\ast}_t(y))$, we have 
\begin{equation*}
\begin{array}{l}
\big[\tnorms{\nabla{d}_t(y) - \widetilde{\nabla}{d_t}(y)}_{y,t}^{\ast}\big]^2 \vspace{1ex}\\
= \frac{M_t^4}{16}(x^{\ast}_t(y) - \widetilde{x}^{\ast}_t(y))A^{\top}\big(\frac{M_t^4}{16}A\nabla^2{\psi_t}(\widetilde{x}^{\ast}_t(y); y)^{-1}A^{\top}\big)^{-1}A(x^{\ast}_t(y) - \widetilde{x}^{\ast}_t(y)) \vspace{1ex}\\
 \leq \frac{M_t^4}{16}(x^{\ast}_t(y) - \widetilde{x}^{\ast}_t(y))^{\top}\frac{16}{M_t^4}\nabla^2{\psi_t}(\widetilde{x}^{\ast}_t(y); y)(x^{\ast}_t(y) - \widetilde{x}^{\ast}_t(y)) \vspace{1ex}\\
 = \abss{x^{\ast}_t(y) - \widetilde{x}^{\ast}_t(y)}_{\widetilde{x}^{\ast}_t(y), t}^2.
\end{array}
\end{equation*}
In the last inequality, we use $A^{\top}(AQ^{-1}A^{\top})^{-1}A \preceq Q$ for any symmetric positive definite matrix $Q$ and any full-row rank matrix $A$.
Hence, we obtain the third estimate of \eqref{eq:inexact_oracle_properties}.
\Eproof

\beforesubsec
\subsection{\bf The technical proofs of Subsection~\ref{subsec:one_iter_analysis}: Convergence analysis}\label{apdx:subsec:one_iter_analysis}
\aftersubsec
We provide the proof of technical results in Subsection~\ref{subsec:one_iter_analysis}.

\beforesubsubsec
\subsubsection{The proof of Lemma~\ref{le:key_property0}: Key estimate of the inexact PN scheme \eqref{eq:inexact_dual_newton_scheme2}.}\label{apdx:le:key_property0}
\aftersubsubsec
For simplicity of presentation, we redefine $t:= t_k$, $t_{+}:= t_{k+1}$, $y:= y^k$, $y_{+}:= y^{k+1}$, and $\bar{y}_{+}:=\bar{y}^{k+1}$, where $\bar{y}^{k+1}$ is defined by \eqref{eq:dual_newton_scheme} or \eqref{eq:dual_newton_scheme2}.
Using these new notations, we also denote  $\delta := \delta_{t_{+}}(y)$, $\delta_{+} := \delta_{t_{+}}(y_{+})$, $\epsilon := \Vert\!\vert y_{+} - \bar{y}_{+} \Vert\!\vert_{y,t_{+}}$, and $\hat{\lambda} := \tnorms{y_{+} - y}_{y, t_{+}}$ to make our analysis more clean.

If we define $r_{t_{+}}(y) := \widetilde{\nabla}^2{d_{t_{+}}}(y)(\bar{y}_{+} - y)  + \widetilde{\nabla}{d_{t_{+}}}(y)$, then from \eqref{eq:dual_newton_scheme}, we have
\begin{equation*}
-r_{t_{+}}(y) \in \partial{h_{t_{+}}}(\bar{y}_{+}),
\end{equation*}
which is equivalent to
\begin{equation*}
\bar{y}_{+} - \widetilde{\nabla}^2{d_{t_{+}}}(y_{+})^{-1}r_{t_{+}}(y) \in \bar{y}_{+} +  \widetilde{\nabla}^2{d_{t_{+}}}(y_{+})^{-1}\partial{h_{t_{+}}}(\bar{y}_{+}).
\end{equation*}
Utilizing the scaled proximal operator defined by \eqref{eq:dual_newton_scheme2}, we can write the last statement as
\begin{equation}\label{eq:lm1_est1}
\bar{y}_{+} = \prox_{h_{t_{+}}}^{\widetilde{\nabla}^2{d_{t_{+}}}(y_{+})}\Big(\bar{y}_{+} - \widetilde{\nabla}^2{d_{t_{+}}}(y_{+})^{-1}r_{t_{+}}(y)\Big).
\end{equation}
Using the definition of $\widetilde{G}_t(y)$ in \eqref{eq:gradient_mapping} and of $\lambda_t(y)$ in \eqref{eq:inexact_NT_decrement}, we can derive
\begin{equation*}
\begin{array}{ll}
\lambda_{t_{+}}(y_{+}) &= \tnorms{\widetilde{G}_{t_{+}}(y_{+})}_{y_{+}, t_{+}}^{\ast} \vspace{1ex}\\
&= \Big\Vert\!\Big\vert \widetilde{\nabla}^2{d_{t_{+}}}(y_{+})\Big[y_{+} - \prox_{h_{t_{+}}}^{\widetilde{\nabla}^2{d_{t_{+}}}(y_{+})}\Big(y_{+} - \widetilde{\nabla}^2{d_{t_{+}}}(y_{+})^{-1}\widetilde{\nabla}d_{t_{+}}(y_{+})\Big)\Big]\Big\Vert\!\Big\vert_{y_{+}, t_{+}}^{\ast} \vspace{1ex}\\
&= \Big\Vert\!\Big\vert y_{+} -   \prox_{h_{t_{+}}}^{\widetilde{\nabla}^2{d_{t_{+}}}(y_{+})}\Big(y_{+} - \widetilde{\nabla}^2{d_{t_{+}}}(y_{+})^{-1}\widetilde{\nabla}d_{t_{+}}(y_{+})\Big)\Big\Vert\!\Big\vert_{y_{+}, t_{+}} \vspace{1ex}\\
&\leq \tnorms{y_{+} - \bar{y}_{+}}_{y_{+}, t_{+}} +  \Big\Vert\!\Big\vert \bar{y}_{+} -   \prox_{h_{t_{+}}}^{\widetilde{\nabla}^2{d_{t_{+}}}(y_{+})}\Big(y_{+} - \widetilde{\nabla}^2{d_{t_{+}}}(y_{+})^{-1}\widetilde{\nabla}d_{t_{+}}(y_{+})\Big)\Big\Vert\!\Big\vert_{y_{+},t_{+}} \vspace{1ex}\\
&\overset{\tiny \tiny\eqref{eq:lm1_est1}}{=} \tnorms{y_{+} - \bar{y}_{+}}_{y_{+},t_{+}} + \Big\Vert\!\Big\vert  \prox_{h_{t_{+}}}^{\widetilde{\nabla}^2{d_{t_{+}}}(y_{+})}\Big(\bar{y}_{+} - \widetilde{\nabla}^2{d_{t_{+}}}(y_{+})^{-1}r_{t_{+}}(y)\Big) \vspace{1ex}\\
& {~~~~~~~~~~~~~~~~~~~~~~~~~~~~~~} -   \prox_{h_{t_{+}}}^{\widetilde{\nabla}^2{d_{t_{+}}}(y_{+})}\Big(y_{+} - \widetilde{\nabla}^2{d_{t_{+}}}(y_{+})^{-1}\widetilde{\nabla}d_{t_{+}}(y_{+})\Big)\Big\Vert\!\Big\vert_{y_{+}, t_{+}}.
\end{array}
\end{equation*}
By the  non-expansiveness of $\prox_{h_{t_{+}}}^{\widetilde{\nabla}^2{d_{t_{+}}}}(\cdot)$, see \cite{Tran-Dinh2013a}, we can further estimate this term as
\begin{equation}\label{eq:lm1_est2}
{\!\!\!\!}\begin{array}{ll}
\lambda_{t_{+}}(y_{+}) &\leq  \tnorms{y_{+} - \bar{y}_{+}}_{y_{+},t_{+}}  +  \Big\Vert\!\Big\vert \bar{y}_{+} - y_{+} + \widetilde{\nabla}^2{d_{t_{+}}}(y_{+})^{-1}\big( \widetilde{\nabla}d_{t_{+}}(y_{+}) - r_{t_{+}}(y)\big) \Big\Vert\!\Big\vert_{y_{+},t_{+}} \vspace{1ex}\\
&\leq 2\tnorms{y_{+} - \bar{y}_{+}}_{y_{+},t_{+}}  + \tnorm{\widetilde{\nabla}d_{t_{+}}(y_{+}) -  \widetilde{\nabla}{d_{t_{+}}}(y) - \widetilde{\nabla}^2{d_{t_{+}}}(y)(\bar{y}_{+} - y)}_{y_{+},t_{+}}^{\ast}.
\end{array}{\!\!\!\!}
\end{equation}
Next, we decompose the following term $R_{t_{+}}(y)$ as 
\begin{equation}\label{eq:lm1_est3}
\begin{array}{ll}
R_{t_{+}}(y) &:= \widetilde{\nabla}d_{t_{+}}(y_{+}) -  \widetilde{\nabla}{d_{t_{+}}}(y) - \widetilde{\nabla}^2{d_{t_{+}}}(y)(\bar{y}_{+} - y) 
\vspace{1ex}\\
&= \big[\widetilde{\nabla}d_{t_{+}}(y_{+}) - \nabla{d}_{t_{+}}(y_{+})\big] - \big[ \widetilde{\nabla}{d_{t_{+}}}(y) - {\nabla}{d_{t_{+}}}(y) \big]\vspace{1ex}\\
&- {~} \big[\widetilde{\nabla}^2{d_{t_{+}}}(y) - \nabla^2{d_{t_{+}}}(y)\big](y_{+} - y) - \widetilde{\nabla}^2{d_{t_{+}}}(y)(\bar{y}_{+} - y_{+}) \vspace{1ex}\\
&+ {~} \big[ \nabla{d}_{t_{+}}(y_{+}) - {\nabla}{d_{t_{+}}}(y) - \nabla^2{d_{t_{+}}}(y)(y_{+} - y) \big].
\end{array}
\end{equation}
Before we estimate the five terms of $R_{t_{+}}(y)$, we recall the following inequalities, which will be repeatedly used in our proof.
\begin{equation}\label{eq:lm1_ineq1}
\frac{1}{1-\norms{y_{+} - y}_{y,t_{+}}} \overset{\tiny \tiny\eqref{eq:inexact_oracle_properties}}{\leq}  \frac{1}{1-\frac{1}{1-\delta_{t_{+}}(y)}\tnorms{y_{+} - y}_{y,t_{+}}}= \frac{1-\delta}{1- \delta - \hat{\lambda}}. 
\end{equation}
\begin{equation}\label{eq:lm1_ineq2}
\begin{array}{ll}
\tnorms{\cdot}_{y_{+},t_{+}}^{\ast} \overset{\tiny \tiny\eqref{eq:inexact_oracle_properties}}{\leq} \frac{\norms{\cdot}_{y_{+},t_{+}}^{\ast}}{1-\delta_{t_{+}}(y_{+})} \leq \frac{\norms{\cdot}_{y,t_{+}}^{\ast}}{\big(1-\delta_{t_{+}}(y_{+})\big)\big(1-\norms{y_{+} - y}_{y,t_{+}}\big)} \overset{\tiny \tiny\eqref{eq:lm1_ineq1}}{\leq} \frac{(1-\delta)\norms{\cdot}_{y,t_{+}}^{\ast}}{(1-\delta_{+})(1- \delta - \hat{\lambda})}. 
\end{array}
\end{equation}
Here, the second last inequality of \eqref{eq:lm1_ineq2} is from \cite[Theoryem 4.1.6]{Nesterov2004}. 
Note that \eqref{eq:lm1_ineq2} also holds for $\tnorms{\cdot}_{y_{+},t_{+}}$ and $\norms{\cdot}_{y,t_{+}}$.

Using  \eqref{eq:lm1_ineq2}, we have
\begin{equation}\label{eq:lm1_ineq3}
\tnorms{\cdot}_{y_{+},t_{+}}^{\ast} \overset{\tiny \tiny\eqref{eq:lm1_ineq2}}{\leq} \frac{(1-\delta)\norms{\cdot}_{y,t_{+}}^{\ast}}{(1-\delta_{+})(1- \delta - \hat{\lambda})} \overset{\tiny \tiny\eqref{eq:inexact_oracle_properties}}{\leq} \frac{\tnorms{\cdot}_{y,t_{+}}^{\ast}}{(1-\delta_{+})(1- \delta - \hat{\lambda})}. 
\end{equation}
Note that \eqref{eq:lm1_ineq3} also holds for $\tnorms{\cdot}_{y_{+},t_{+}}$ and $\tnorms{\cdot}_{y,t_{+}}$.

Now, we estimate the first term in $R_{t_{+}}(y)$ of \eqref{eq:lm1_est3} as
\begin{equation}\label{eq:lm1_estR1}
\tnorms{\widetilde{\nabla}d_{t_{+}}(y_{+}) - \nabla{d}_{t_{+}}(y_{+})}_{y_{+},t_{+}}^{\ast} \overset{\tiny \tiny\eqref{eq:inexact_oracle_properties}}{\leq} \delta_{t_{+}}(y_{+}) = \delta_{+}. 
\end{equation}
For the second term of  \eqref{eq:lm1_est3}, we have
\begin{equation}\label{eq:lm1_estR2}
\begin{array}{ll}
\tnorms{\widetilde{\nabla}d_{t_{+}}(y) - \nabla{d}_{t_{+}}(y)}_{y_{+},t_{+}}^{\ast} &\overset{\tiny \tiny\eqref{eq:lm1_ineq3}}{\leq} \frac{1}{(1-\delta_{+})(1- \delta - \hat{\lambda})} \tnorms{\widetilde{\nabla}d_{t_{+}}(y) - \nabla{d}_{t_{+}}(y)}_{y,t_{+}}^{\ast}  \vspace{1ex}\\
&\overset{\tiny \tiny\eqref{eq:inexact_oracle_properties}}{\leq} \frac{\delta_{t_{+}}(y)}{(1-\delta_{+})(1- \delta - \hat{\lambda})} \vspace{1ex}\\
&= \frac{\delta}{(1 - \delta_{+})(1-\delta - \hat{\lambda})}.
\end{array}
\end{equation}
To estimate the third term of  \eqref{eq:lm1_est3}, let $S(y) := \big[\widetilde{\nabla}^2{d_{t_{+}}}(y) - \nabla^2{d_{t_{+}}}(y)\big](y_{+} - y)$.
We have
\begin{equation}\label{eq:lm1_estR3a}
\begin{array}{l}
\tnorms{S(y)}_{y_{+},t_{+}}^{\ast} \overset{\tiny \tiny\eqref{eq:lm1_ineq2}}{\leq}  \frac{(1-\delta)\norms{S(y)}_{y,t_{+}}^{\ast}}{(1 - \delta_{+})(1 -  \delta - \hat{\lambda})}
\end{array}
\end{equation}
However, $\norms{S(y)}_{y,t_{+}}^{\ast}$ can be estimated as
\begin{equation*}
{\!\!\!}\begin{array}{ll}
\left[\norms{S(y)}^{\ast}_{y, t_{+}}\right]^2 &=(y_{+} - y)^{\top}[\widetilde{\nabla}^2{d_{t_{+}}}(y) - \nabla^2{d_{t_{+}}}(y)\big]
\nabla^2{d_{t_{+}}}(y)^{-1}
[\widetilde{\nabla}^2{d_{t_{+}}}(y) - \nabla^2{d_{t_{+}}}(y)\big](y_{+} - y)\\
&= (y_{+} {\!} - y)^{\top}\widetilde{\nabla}^2{d_{t_{+}}}(y)\nabla^2{d_{t_{+}}}(y)^{-1}\widetilde{\nabla}^2{d_{t_{+}}}(y)(y_{+} {\!} - y)\vspace{1ex}\\
&{~~} - {~} 2(y_{+} - y)^{\top}\widetilde{\nabla}^2{d_{t_{+}}}(y)(y_{+} - y) + (y_{+} - y)^{\top}\nabla^2{d_{t_{+}}}(y)(y_{+} - y)\vspace{1ex}\\
&\overset{\tiny \eqref{eq:inexact_oracle_properties}}{\leq} \left[\frac{2}{(1-\delta_{t_{+}}(y))^2} - 2\right](y_{+} - y)^{\top}\widetilde{\nabla}^2{d_{t_{+}}}(y)(y_{+} - y)\vspace{1ex}\\
&= \frac{4\delta-2\delta^2}{(1-\delta)^2}\tnorms{y_{+} - y}_{y,t_{+}}^2 = \frac{4\delta-2\delta^2}{(1-\delta)^2}\hat{\lambda}^2
\end{array}{\!\!\!\!}
\end{equation*}
Using this estimate into \eqref{eq:lm1_estR3a}, we finally get
\begin{equation}\label{eq:lm1_estR3b}
\begin{array}{ll}
\tnorms{\big[\widetilde{\nabla}^2{d_{t_{+}}}(y) - \nabla^2{d_{t_{+}}}(y)\big](y_{+} - y)}_{y_{+}, t_{+}}^{\ast} &\leq \frac{1-\delta}{(1 - \delta_{+})(1 - \hat{\lambda} - \delta)}\frac{\sqrt{4\delta-2\delta^2}}{(1-\delta)}\hat{\lambda}\vspace{1ex}\\ & =\frac{\sqrt{4\delta-2\delta^2}\hat{\lambda}}{(1 - \delta_{+})(1 - \hat{\lambda} - \delta)}.
\end{array}
\end{equation}
For the fourth term $\widetilde{\nabla}^2{d_{t_{+}}}(y)(\bar{y}_{+} - y_{+})$ of  \eqref{eq:lm1_est3}, we have
\begin{equation}\label{eq:lm1_estR4}
\begin{array}{ll}
\tnorms{\widetilde{\nabla}^2{d_{t_{+}}}(y)(\bar{y}_{+} - y_{+})}_{y_{+}, t_{+}}^{\ast} &\overset{\tiny \tiny\eqref{eq:lm1_ineq3}}{\leq} \frac{1}{(1 -\delta_{+})(1- \hat{\lambda} - \delta)}\tnorm{\widetilde{\nabla}^2{d_{t_{+}}}(y)(\bar{y}_{+} - y_{+})}_{y, t_{+}}^{\ast}\vspace{1ex}\\
 &= \frac{1}{(1 -\delta_{+})(1 - \hat{\lambda} - \delta)}\tnorm{\bar{y}_{+} - y_{+}}_{y, t_{+}}\vspace{1ex}\\
 &= \frac{\epsilon}{(1 -\delta_{+})(1- \hat{\lambda} - \delta)}.
\end{array}
\end{equation}
Finally, we estimate last terms $\Tc_5 := \tnorms{\nabla{d}_{t_{+}}(y_{+}) - {\nabla}{d_{t_{+}}}(y) - \nabla^2{d_{t_{+}}}(y)(y_{+} - y)}_{y_{+},t_{+}}^{\ast}$.
Note that
\begin{equation}\label{eq:lm1_estR5}
\begin{array}{ll}
\Tc_5 &\overset{\tiny \tiny\eqref{eq:lm1_ineq2}}{\leq} \frac{1-\delta}{(1-\delta_{+})(1-\delta-\hat{\lambda})}\left\Vert \big[ \nabla{d}_{t_{+}}(y_{+}) - {\nabla}{d_{t_{+}}}(y) - \nabla^2{d_{t_{+}}}(y)(y_{+} - y)\right\Vert_{y,t_{+}}^{\ast}\vspace{1ex}\\
&\leq \frac{1-\delta}{(1 - \delta_{+})(1 - \delta_{+} - \hat{\lambda})}\left( \frac{\norm{y_{+} - y}_{y, t_{+}}^2}{1 - \norm{y_{+} - y}_{y,t_{+}}}\right) \vspace{1ex}\\
&\overset{\tiny \tiny\eqref{eq:lm1_ineq1}}{\leq} \frac{(1 - \delta)^2}{(1 - \delta_{+})(1 - \delta - \hat{\lambda})^2}\norm{y_{+} - y}_{y, t_{+}}^2 \vspace{1ex}\\
&\overset{\tiny \tiny\eqref{eq:inexact_oracle_properties}}{\leq} \frac{\hat{\lambda}^2}{(1 - \delta_{+})(1-\delta - \hat{\lambda})^2},
\end{array}
\end{equation}
where the second inequality follows from \cite[Theorem 1]{TranDinh2016c}.

Plugging \eqref{eq:lm1_estR1}, \eqref{eq:lm1_estR2}, \eqref{eq:lm1_estR3b}, \eqref{eq:lm1_estR4}, and \eqref{eq:lm1_estR5} into \eqref{eq:lm1_est3}, we can estimate
\begin{equation}\label{eq:lm1_est4}
\begin{array}{ll}
\tnorms{R_{t_{+}}(y)}_{y_{+}, t_{+}}^{\ast}
&\leq \delta_{+}  + \frac{\delta}{(1 - \delta_{+})(1- \hat{\lambda} - \delta)} + \frac{\sqrt{4\delta - 2\delta^2}\hat{\lambda}}{(1 - \delta_{+})(1 - \hat{\lambda} - \delta)} \vspace{1ex}\\
&+ {~} \frac{\epsilon}{(1 -\delta_{+})(1 - \hat{\lambda} - \delta)}	 + \frac{\hat{\lambda}^2}{(1 - \delta_{+})(1-\delta- \hat{\lambda})^2}.
\end{array}
\end{equation}
Note that 
\begin{equation*}
\tnorms{y_{+} - \bar{y}_{+}}_{y_{+}, t_{+}} \overset{\tiny \tiny\eqref{eq:lm1_ineq3}}{\leq} \frac{\tnorm{y_{+} - \bar{y}_{+}}_{y, t_{+}}}{(1-\delta_{+})(1-\delta - \hat{\lambda})} = \frac{\epsilon}{(1-\delta_{+})(1-\delta- \hat{\lambda})}.
\end{equation*}
Substituting this estimate and \eqref{eq:lm1_est4} into \eqref{eq:lm1_est2}, we finally obtain
\begin{equation}\label{eq:lm1_est5}
\begin{array}{ll}
\lambda_{t_{+}}(y_{+}) &\leq \frac{3\epsilon}{(1-\delta_{+})(1-\delta - \hat{\lambda})}  + \delta_{+}   + \frac{\delta}{(1 - \delta_{+})(1 - \hat{\lambda} - \delta)} \vspace{1ex}\\
&+ {~} \frac{\sqrt{4\delta - 2\delta^2}\hat{\lambda}}{(1 - \delta_{+})(1 - \hat{\lambda} - \delta)}
+ \frac{\hat{\lambda}^2}{(1 - \delta_{+})(1-\delta - \hat{\lambda})^2}.
\end{array}
\end{equation}
However, from the definition of $\lambda_{t_{+}}(y)$, we have
\begin{equation*} 
\begin{array}{ll}
\lambda_{t_{+}}(y) &:=  \Big\Vert\!\Big\vert\widetilde{\nabla}^2{d_{t_{+}}}(y)\left( y - \text{prox}_{h_{t_{+}}}^{\widetilde{\nabla}^2{d_{t_{+}}}(y)}(y-\widetilde{\nabla}^2{d_{t_{+}}}(y)^{-1}\widetilde{\nabla}{d_{t_{+}}}(y)\right)\Big\Vert\!\Big\vert_{y, t_{+}}^{\ast} \vspace{1ex}\\
& = \tnorm{\bar{y}_{+} - y}_{y, t_{+}} \vspace{1ex}\\
&\geq \tnorms{y_{+} - y}_{y, t_{+}} - \tnorm{y_{+} - \bar{y}_{+}}_{y, t_{+}} \vspace{1ex}\\
& = \hat{\lambda} - \epsilon.
\end{array}
\end{equation*}
This implies $\hat{\lambda} := \tnorms{y_{+} - y}_{y,t_{+}}  \leq \lambda_{t_{+}}(y) {\,} + {\,} \epsilon$.
Substituting this estimate into \eqref{eq:lm1_est5}, we obtain \eqref{eq:key_est00}.
In particular, if $\delta = \delta_{+} = \epsilon = 0$, then we can simplify \eqref{eq:key_est00} to obtain \eqref{eq:key_est01}.
\Eproof

\beforesubsubsec
\subsubsection{The proof of Lemma~\ref{le:key_property1}: The relationship between $\lambda_{t_{+}}(y)$ and $\tilde{\Delta}$.}\label{apdx:le:key_property1}
\aftersubsubsec
We again redefine  $t:= t_k$, $t_{+}:= t_{k+1}$, $y:= y^k$, $y_{+}:= y^{k+1}$, and $\bar{y}_{+}:=\bar{y}^{k+1}$ as in Lemma \ref{le:key_property0}. 
In addition, we also define $\bar{u} := \prox_{h_t}^{\widetilde{\nabla}^2{d_{t}}(y)}\big( y - \widetilde{\nabla}^2{d_{t}}(y)^{-1}\widetilde{\nabla}{d_t}(y)\big)$.

First, we show that $\lambda_{t_{+}}(y)$ and $\lambda_t(y)$ can be respectively expressed as
\begin{equation}\label{eq:lm2_def}
\begin{array}{ll}
\lambda_{t_{+}}(y)&:=\Big\Vert\!\Big\vert \widetilde{\nabla}^2{d_{t_{+}}}(y)\left( y - \text{prox}_{h_{t_{+}}}^{\widetilde{\nabla}^2{d_{t_{+}}}(y)}(y-\widetilde{\nabla}^2{d_{t_{+}}}(y)^{-1}\widetilde{\nabla}{d_{t_{+}}}(y)\right) \Big\Vert\!\Big\vert_{y,t_{+}}^{\ast} \vspace{1ex}\\
& = \tnorms{\bar{y}_{+} - y}_{y,t_{+}}, \vspace{1ex}\vspace{1ex}\\
\lambda_{t}(y)&:=\Big\Vert\!\Big\vert \widetilde{\nabla}^2{d_{t}}(y)\left( y - \text{prox}_{h_t}^{\widetilde{\nabla}^2{d_{t}}(y)}(y-\widetilde{\nabla}^2{d_{t}}(y)^{-1}\widetilde{\nabla}{d_{t}}(y)\right) \Big\Vert\!\Big\vert_{y,t}^{\ast} \vspace{1ex}\\
& = \tnorms{\bar{u} - y}_{y,t}.
\end{array}
\end{equation}
If we denote by $\bar{h}(y) := \phi^{\ast}(-y)$, then $h_t(y) = \frac{M_t^2}{4}\bar{h}(y)$.
By the definition of $\bar{u}$ and $\bar{y}_{+}$, we can write 
\begin{equation*}
\left\{\begin{array}{ll}
-\frac{4}{M_t^2}\big[\widetilde{\nabla}^2{d_{t}}(y)(\bar{u} - y) + \widetilde{\nabla}{d_{t}}(y) \big] &\in \partial{\bar{h}}(\bar{u}),\vspace{1ex}\\
- \frac{4}{M_{t_{+}}^2}\big[\widetilde{\nabla}^2{d_{t_{+}}}(y)(\bar{y}_{+} - y) + \widetilde{\nabla}{d_{t_{+}}}(y) \big] &\in \partial{\bar{h}}(\bar{y}_{+}).
\end{array}\right.
\end{equation*}
Using the monotonicity of $\partial{\bar{h}}(\cdot)$, we can show that
\begin{equation*}
\iprod{\tfrac{4}{M_t^2}\widetilde{\nabla}^2{d_{t}}(y)(\bar{u} - y) - \tfrac{4}{M_{t_{+}}^2}\widetilde{\nabla}^2{d_{t_{+}}}(y)(\bar{y}_{+} - y) +  \tfrac{4}{M_{t}^2}\widetilde{\nabla}{d_{t}}(y)  -  \tfrac{4}{M_{t_{+}}^2}\widetilde{\nabla}{d_{t_{+}}}(y), \bar{y}_{+} - \bar{u}} \geq 0.
\end{equation*}
Rearranging this inequality, we obtain
\begin{equation*}
\begin{array}{ll}
&\iprod{\tfrac{4}{M_t^2}\widetilde{\nabla}^2{d_{t}}(y)(\bar{u} - y) - \tfrac{4}{M_{t_{+}}^2}\widetilde{\nabla}^2{d_{t_{+}}}(y)(\bar{u} - y) +  \tfrac{4}{M_{t}^2}\widetilde{\nabla}{d_{t}}(y)  -  \tfrac{4}{M_{t_{+}}^2}\widetilde{\nabla}{d_{t_{+}}}(y), \bar{y}_{+} - \bar{u}} \vspace{1ex}\\
& {~~~~~~}\geq \tfrac{4}{M_{t_{+}}^2}\tnorms{\bar{y}_{+} - \bar{u}}_{y, t_{+}}^2. 
\end{array}
\end{equation*}
By the Cauchy-Schwarz inequality, we can derive that
\begin{equation}\label{eq:lm2_est1}
\begin{array}{ll}
&\Big\Vert\!\Big\vert \overbrace{\tfrac{4}{M_t^2}\widetilde{\nabla}^2{d_{t}}(y)(\bar{u} - y) - \tfrac{4}{M_{t_{+}}^2}\widetilde{\nabla}^2{d_{t_{+}}}(y)(\bar{u} - y)}^{\Tc_1} +  \overbrace{\tfrac{4}{M_{t}^2}\widetilde{\nabla}{d_{t}}(y)  -  \tfrac{4}{M_{t_{+}}^2}\widetilde{\nabla}{d_{t_{+}}}(y)}^{\Tc_2}\Big\Vert\!\Big\vert_{y, t_{+}}^{\ast} \vspace{1ex}\\
&{~~~~~~~} \geq \tfrac{4}{M_{t_{+}}^2}\tnorms{\bar{y}_{+} - \bar{u}}_{y, t_{+}}.
\end{array}
\end{equation}
To estimate $\Tc_1$, we first show the relationship between $\nabla^2{\psi_t}(\widetilde{x}^{\ast}_t(y))$ and $\nabla^2{\psi_{t_{+}}}(\widetilde{x}^{\ast}_{t_{+}}(y))$. 
Then, we use it to get the relationship between $\widetilde{\nabla}^2{d_{t}}(y)$ and $\widetilde{\nabla}^2{d_{t_{+}}}(y)$. 
Recall that $\nabla^2{\psi_t}(x) := \frac{M_t^2}{4}\big[\nabla^2{g}(x) + t\nabla^2{f}(x)\big]$. 
Moreover,  if $t\in [0, 1]$, then $\frac{M_t^2}{4} := \max\set{1,\frac{1}{t}} = \frac{1}{t}$.
Therefore, we can write
\begin{equation*}
\nabla^2{\psi_{t}}(x) = \frac{1}{t}\nabla^2{g}(x)+\nabla^2{f}(x)~~\text{and}~~~\nabla^2{\psi_{t_{+}}}(x) = \frac{1}{t_{+}}\nabla^2{g}(x)+\nabla^2{f}(x).
\end{equation*}
For any $0 \leq t_{+} \leq t \leq 1$, we have
\begin{equation}\label{eq:lm2_est_psi1}
\nabla^2{\psi_{t}}(x) \preceq \nabla^2{\psi_{t_{+}}}(x) \preceq \frac{1}{t_{+}}\nabla^2{g}(x)+\frac{t}{t_{+}}\nabla^2{f}(x) = \frac{t}{t_{+}}\nabla^2{\psi_{t}}(x).
\end{equation}
In addition, using the self-concordance of $\psi_t$ and \eqref{eq:lm2_est_psi1}, we also have
\begin{equation}\label{eq:lm2_est_psi2}
\begin{array}{ll}
\nabla^2{\psi_t}(\widetilde{x}^{\ast}_t(y)) &\preceq \tfrac{1}{\big(1 {~} - {~} \abss{\widetilde{x}^{\ast}_{t_{+}}(y) - \widetilde{x}^{\ast}_t(y)}_{\widetilde{x}^{\ast}_t(y),t}\big)^2}\nabla^2{\psi_t}(\widetilde{x}^{\ast}_{t_{+}}(y)) \vspace{1ex}\\
& \overset{\tiny \tiny\eqref{eq:lm2_est_psi1}}{\preceq} \frac{1}{(1 - \tilde{\Delta})^2}\nabla^2{\psi_{t_{+}}}(\widetilde{x}^{\ast}_{t_{+}}(y)),\vspace{1ex}\\
\nabla^2{\psi_{t_{+}}}(\widetilde{x}^{\ast}_{t_{+}}(y)) & \overset{\tiny \tiny\eqref{eq:lm2_est_psi1}}{\preceq} \frac{t}{t_{+}}\nabla^2{\psi_{t}}(\widetilde{x}^{\ast}_{t_{+}}(y))  \preceq \tfrac{t}{t_{+}\big(1 {~}- {~} \abss{\widetilde{x}^{\ast}_{t_{+}}(y) - \widetilde{x}^{\ast}_t(y)}_{\widetilde{x}^{\ast}_{t}(y),t}\big)^2}\nabla^2{\psi_{t}}(\widetilde{x}^{\ast}_{t}(y)) \vspace{1ex}\\
& = \frac{t}{t_{+}(1-\tilde{\Delta})^2}{}\nabla^2{\psi_{t}}(\widetilde{x}^{\ast}_{t}(y)). 
\end{array}
\end{equation}
If we take the inverse of both sides of \eqref{eq:lm2_est_psi2}, then we get
\begin{equation*} 
\left\{\begin{array}{ll}
\nabla^2{\psi_{t_{+}}}(\widetilde{x}^{\ast}_{t_{+}}(y))^{-1} & \preceq \frac{1}{(1 - \tilde{\Delta})^2}\nabla^2{\psi_t}(\widetilde{x}^{\ast}_t(y))^{-1},\vspace{1.25ex}\\
\nabla^2{\psi_{t}}(\widetilde{x}^{\ast}_{t}(y))^{-1} & \preceq \frac{t}{t_{+}(1-\tilde{\Delta})^2}\nabla^2{\psi_{t_{+}}}(\widetilde{x}^{\ast}_{t_{+}}(y))^{-1}. 
\end{array}\right.
\end{equation*}
Since $\widetilde{\nabla}^2{d_{t}}(y) = \frac{M_t^4}{16}A\nabla^2{\psi_t}(\widetilde{x}^{\ast}_t(y))^{-1}A^{\top}$ and $\widetilde{\nabla}^2{d_{t_{+}}}(y) = \frac{M_{t_{+}}^4}{16}A\nabla^2{\psi_{t_{+}}}(\widetilde{x}^{\ast}_{t_{+}}(y))^{-1}A^{\top}$, the last inequalities imply
\begin{equation*}
\begin{array}{ll}
\frac{16}{M_{t_{+}}^4}\widetilde{\nabla}^2{d_{t_{+}}}(y)  \preceq \frac{16}{M_t^4(1 - \tilde{\Delta})^2} \widetilde{\nabla}^2{d_{t}}(y) ~~~\text{and}~~~\frac{16}{M_t^4}\widetilde{\nabla}^2{d_{t}}(y)  \preceq \frac{16t}{M_{t_{+}}^4t_{+}(1-\tilde{\Delta})^2}\widetilde{\nabla}^2{d_{t_{+}}}(y),
\end{array}
\end{equation*}
which are respectively equivalent to
\begin{equation*}
\begin{array}{ll}
\widetilde{\nabla}^2{d_{t_{+}}}(y)  \preceq \frac{M_{t_{+}}^4}{M_t^4(1 - \tilde{\Delta})^2} \widetilde{\nabla}^2{d_{t}}(y) ~~~~\text{and}~~~~\widetilde{\nabla}^2{d_{t}}(y)  \preceq \frac{M_t^4t}{M_{t_{+}}^4t_{+}(1-\tilde{\Delta})^2}\widetilde{\nabla}^2{d_{t_{+}}}(y) .
\end{array}
\end{equation*}
Since $\frac{M_t^2}{4}=\frac{1}{t}$ and $\frac{M_{t_{+}}^2}{4}=\frac{1}{t_{+}}$, we obtain from the above inequalities that 
\begin{equation}\label{eq:lm2_est_d}
\begin{array}{ll}
\widetilde{\nabla}^2{d_{t_{+}}}(y)  \preceq \frac{t^2}{t_{+}^2(1 - \tilde{\Delta})^2} \widetilde{\nabla}^2{d_{t}}(y)~~~~\text{and}~~~~\widetilde{\nabla}^2{d_{t}}(y)  \preceq \frac{t_{+}}{t(1-\tilde{\Delta})^2}\widetilde{\nabla}^2{d_{t_{+}}}(y) .
\end{array}
\end{equation}
Now we can estimate the first term $\Tc_1$ in \eqref{eq:lm2_est1} as
\begin{equation}\label{eq:lm2_est2}
{\!\!\!\!\!}\begin{array}{ll}
\big[\tnorms{\Tc_1}_{y,t_{+}}^{\ast}{\!}\big]^2 {\!\!}&= \Big[\big\Vert\!\big\vert \tfrac{4}{M_t^2}\widetilde{\nabla}^2{d_{t}}(y)(\bar{u} - y) - \tfrac{4}{M_{t_{+}}^2}\widetilde{\nabla}^2{d_{t_{+}}}(y)(\bar{u} - y)\big\Vert\!\big\vert_{y, t_{+}}^{\ast}\Big]^2\vspace{1ex}\\
&{\!\!\!} = \Big[\big\Vert\!\big\vert t\widetilde{\nabla}^2{d_{t}}(y)(\bar{u} - y) - t_{+}\widetilde{\nabla}^2{d_{t_{+}}}(y)(\bar{u} - y)\big\Vert\!\big\vert_{y, t_{+}}^{\ast}\Big]^2\vspace{1ex}\\
&{\!\!\!} = (\bar{u} \!-\! y)^{\top}{\!}\Big(\big[t\widetilde{\nabla}^2{d_{t}}(y) \!-\! t_{+}\widetilde{\nabla}^2{d_{t_{+}}}(y)\big]\widetilde{\nabla}^2{d_{t_{+}}}(y)^{-1}\big[t\widetilde{\nabla}^2{d_{t}}(y) \!-\! t_{+}\widetilde{\nabla}^2{d_{t_{+}}}(y)\big]\Big)(\bar{u} \!-\! y) \vspace{1ex}\\
&{\!\!\!} = (\bar{u} - y)^{\top}\Big(t^2\widetilde{\nabla}^2{d_{t}}(y)\widetilde{\nabla}^2{d_{t_{+}}}(y)^{-1}\widetilde{\nabla}^2{d_{t}}(y)-2tt_{+}\widetilde{\nabla}^2{d_{t}}(y)+t_{+}^2\widetilde{\nabla}^2{d_{t_{+}}}(y)\Big)(\bar{u} - y) \vspace{1ex}\\
&\overset{\tiny \eqref{eq:lm2_est_d}}{\leq}
(\bar{u} - y)^{\top}\Big(\frac{t_{+}t}{(1-\tilde{\Delta})^2}-2tt_{+}+\frac{t^2}{(1-\tilde{\Delta})^2}\widetilde{\nabla}^2{d_{t}}(y)\Big)(\bar{u} - y).\vspace{1ex}\\
&=\frac{t^2-2t_{+}t(1-\tilde{\Delta})^2+tt_{+}}{(1-\tilde{\Delta})^2}\tnorms{\bar{u} - y}_{y, t}^2.
\end{array}{\!\!\!}
\end{equation}
To estimate the second term $\Tc_2$ of \eqref{eq:lm2_est1}, by the definition of $\widetilde{\nabla}{d_t}$, we have
\begin{equation}\label{eq:lm2_est3}
\begin{array}{ll}
\big[\tnorms{\Tc_2}_{y, t_{+}}^{\ast}{\!}\big]^2 {\!}&=  \Big[\big\Vert\!\big\vert \tfrac{4}{M_{t}^2}\widetilde{\nabla}{d_{t}}(y)  -  \tfrac{4}{M_{t_{+}}^2}\widetilde{\nabla}{d_{t_{+}}}(y) \big\Vert\!\big\vert_{y,t_+}^{\ast}\Big]^2 \vspace{1ex}\\
& = \Big[\big\Vert\!\big\vert A\tilde{x}_{t}^{\ast}(y)  -  A\tilde{x}_{t_{+}}^{\ast}(y)\big\Vert\!\big\vert_{y, t_{+}}^{\ast}\Big]^2\vspace{1ex}\\
& = (\widetilde{x}^{\ast}_t(y) - \tilde{x}_{t_{+}}^{\ast}(y))^{\top}A^{\top}\widetilde{\nabla}^2{d_{t_{+}}}(y)^{-1}A(\widetilde{x}^{\ast}_t(y) - \tilde{x}_{t_{+}}^{\ast}(y)) \vspace{1ex}\\
&\overset{\tiny \tiny\eqref{eq:inexact_oracle}}{=} \frac{16}{M_{t_{+}}^4}(\widetilde{x}^{\ast}_t(y) {\!} - {\!} \tilde{x}_{t_{+}}^{\ast}(y))^{\top}A^{\top}{\!\!}\left(A\nabla^2{\psi_{t_{+}}}(\widetilde{x}^{\ast}_{t_{+}}(y))^{-1}{\!\!}A^{\top}\right)^{-1}{\!\!}A(\widetilde{x}^{\ast}_t(y) - \tilde{x}_{t_{+}}^{\ast}(y)) \vspace{1ex}\\
&\leq \frac{16}{M_{t_{+}}^4}(\widetilde{x}^{\ast}_t(y) - \tilde{x}_{t_{+}}^{\ast}(y))^{\top}\nabla^2{\psi_{t_{+}}}(\widetilde{x}^{\ast}_{t_{+}}(y))(\widetilde{x}^{\ast}_t(y) - \tilde{x}_{t_{+}}^{\ast}(y)) \vspace{1ex}\\
&= \frac{16}{M_{t_{+}}^4}\abss{\widetilde{x}^{\ast}_t(y) - \tilde{x}_{t_{+}}^{\ast}(y))}_{\tilde{x}_{t_{+}}^{\ast}(y), t_{+}}^2 = \frac{16}{M_{t_{+}}^4}\tilde{\Delta}_{+}^2.
\end{array}
\end{equation}
Here, we use the fact that $A^{\top}(AQ^{-1}A^{\top})^{-1}A \preceq Q$ for any symmetric positive definite matrix $Q$ and any full-row rank matrix $A$.

Plugging \eqref{eq:lm2_est2} and \eqref{eq:lm2_est3} into \eqref{eq:lm2_est1}, we get
\begin{equation*} 
\tfrac{4}{M_{t_{+}}^2}\tnorms{\bar{y}_{+} - \bar{u}}_{y, t_{+}} \leq \tfrac{4}{M_{t_{+}}^2}\tilde{\Delta}_{+} + \frac{\sqrt{t^2-2t_{+}t(1-\tilde{\Delta})^2+tt_{+}}}{1-\tilde{\Delta}}\tnorms{\bar{u} - y}_{y,t}.
\end{equation*}
This inequality is equivalent to 
\begin{equation}\label{eq:lm2_est4} 
\tnorms{\bar{y}_{+} - \bar{u}}_{y, t_{+}} \leq \tilde{\Delta}_{+} + \frac{\sqrt{(\frac{t}{t_{+}})^2-2\frac{t}{t_{+}}(1-\tilde{\Delta})^2+\frac{t}{t_{+}}}}{1-\tilde{\Delta}}\tnorms{\bar{u} - y}_{y,t}.
\end{equation}
Finally, we can derive
\begin{equation}\label{eq:lm2_est5} 
\begin{array}{ll}
\lambda_{t_{+}}(y) & \overset{\tiny \tiny\eqref{eq:lm2_def}}{=} \tnorms{\bar{y}_{+} - y}_{y, t_{+}} \vspace{1ex}\\
& \leq \tnorms{\bar{y}_{+} - \bar{u}}_{y, t_{+}} + \tnorms{\bar{u} - y}_{y, t_{+}} \vspace{1ex}\\
&\overset{\tiny \tiny\eqref{eq:lm2_est4}}{\leq} \tilde{\Delta}_{+} + \frac{\sqrt{(\frac{t}{t_{+}})^2-2\frac{t}{t_{+}}(1-\tilde{\Delta})^2+\frac{t}{t_{+}}}}{1-\tilde{\Delta}}\lambda_t(y) + \tnorms{\bar{u} - y}_{y, t_{+}} \vspace{1ex}\\
& \overset{\tiny \tiny\eqref{eq:lm2_est_d}}{\leq} \tilde{\Delta}_{+} + \frac{\sqrt{(\frac{t}{t_{+}})^2-2\frac{t}{t_{+}}(1-\tilde{\Delta})^2+\frac{t}{t_{+}}}}{1-\tilde{\Delta}}\lambda_t(y) + \frac{t}{t_{+}(1-\tilde{\Delta})}\tnorms{\bar{u} - y}_{y, t} \vspace{1ex}\\
&\overset{\tiny \tiny\eqref{eq:lm2_def}}{\leq} \tilde{\Delta}_{+} + \Bigg[\frac{\sqrt{(\frac{t}{t_{+}})^2-2\frac{t}{t_{+}}(1-\tilde{\Delta})^2+\frac{t}{t_{+}}}~+~\frac{t}{t_{+}}}{1-\tilde{\Delta}}\Bigg]\lambda_t(y),
\end{array}
\end{equation}
which is exactly \eqref{eq:key_est02} due to the update $t_{+} := \sigma t$.
\Eproof

\vspace{2ex}
In order to prove Lemma~\ref{le:key_property2} we need the following auxiliary result.

\begin{lemma}\label{le:key_property3}
Let $a\in (0,1)$ and $b\in(0,1)$ be two positive numbers such that $a+b <1$.
Let
\begin{equation*}
\Nc(a, b) := \set{ (u, v) \in \R^2_{+} ~\mid~ \frac{u^2}{1+u}\leq au + bv,  ~\frac{v^2}{1+v}\leq av + bu}.
\end{equation*}
Then, $\Nc(a,b) \subseteq \set{(u, v) \in \R^2_{+} ~\mid~ u\leq \frac{a+b}{1-a-b}, ~v\leq \frac{a+b}{1-a-b}}$.
\end{lemma}

\begin{proof}
Suppose $(u, v)\in \Nc(a,b)$ and $u > \frac{a+b}{1-a-b}$.
Then, according to $\frac{u^2}{1 + u}\leq au + bv$, we have
\begin{equation}\label{eq:lm4_est1}
v\geq  \frac{1}{b}\left(\frac{u^2}{1 + u} - au \right) =\frac{u}{b}\left(\frac{u}{1 + u} - a \right) > \frac{u}{b}\left(\frac{\frac{a+b}{1-a-b}}{1 + \frac{a+b}{1-a-b}}-a\right) = u  > \frac{a+b}{1-a-b}.
\end{equation}
Therefore, we can show that
\begin{equation}\label{eq:lm4_est2}
\frac{1}{b}\left(\frac{v^2}{1 + v} - av \right) =\frac{v}{b}\left(\frac{v}{1 + v} - a\right) \overset{\tiny\eqref{eq:lm4_est1}}{>} \frac{v}{b}\left(\frac{\frac{a+b}{1-a-b}}{1+\frac{a+b}{1-a-b}} - a \right) = v.
\end{equation}
However, because $\frac{v^2}{1 + v}\leq av + bu$, one can show that
\begin{equation*} 
u \geq\frac{1}{b}\left(\frac{v^2}{1 + v} - av \right) \overset{\tiny\eqref{eq:lm4_est2}}{>} v.
\end{equation*}
This contradicts \eqref{eq:lm4_est1}. 
Consequently, we must have $u\leq\frac{a+b}{1-a-b}$. 
Use the symmetry between $u$ and $v$, we also have $v\leq\frac{a+b}{1-a-b}$.
\Eproof
\end{proof}

\beforesubsubsec
\subsubsection{The proof of Lemma~\ref{le:key_property2}: Upper bound on the solution difference $\tilde{\Delta}$.}\label{apdx:le:key_property2}
\aftersubsubsec
First, by the self-concordance of $\psi_t$, we have
\begin{equation}\label{eq:lm3_est1}
\begin{array}{ll}
\frac{\tilde{\Delta}^2}{1 + \tilde{\Delta}} &\leq \iprods{\widetilde{x}^{\ast}_{t_{+}}(y) - \widetilde{x}^{\ast}_{t}(y), \nabla{\psi_t}(\widetilde{x}^{\ast}_{t_{+}}(y)) - \nabla{\psi_t}(\widetilde{x}^{\ast}_{t}(y))} \vspace{1ex}\\
&\leq \tilde{\Delta}\abss{\nabla{\psi_t}(\widetilde{x}^{\ast}_{t}(y))}_{\widetilde{x}^{\ast}_t(y), t}^{\ast} + \tilde{\Delta}_{+} \abss{\nabla{\psi_t}(\widetilde{x}^{\ast}_{t_{+}}(y))}_{\widetilde{x}^{\ast}_{t_{+}}(y), t_{+}}^{\ast}.
\end{array}
\end{equation}
Next, since $\nabla{\psi_{t}}(x) = \frac{1}{t}\nabla{g}(x)+\nabla{f}(x)$ and  $\nabla{\psi_{t_{+}}}(x) = \frac{1}{t_{+}}\nabla{g}(x)+\nabla{f}(x)$, we have 
\begin{equation*}
\nabla{\psi_t}(\widetilde{x}^{\ast}_{t_{+}}(y)) = \frac{t_{+}}{t}\nabla{\psi_{t_{+}}}(\widetilde{x}^{\ast}_{t_{+}}(y)) + \frac{(t-t_{+})}{t}\nabla{f}(\widetilde{x}^{\ast}_{t_{+}}(y)). 
\end{equation*}
Therefore, we can bound
\begin{equation}\label{eq:lm3_est2}
\begin{array}{ll}
 \abss{\nabla{\psi_t}(\widetilde{x}^{\ast}_{t_{+}}(y))}_{\widetilde{x}^{\ast}_{t_{+}}(y), t_{+}}^{\ast} &\leq \frac{t_{+}}{t} \abss{\nabla{\psi_{t_{+}}}(\widetilde{x}^{\ast}_{t_{+}}(y))}_{\widetilde{x}^{\ast}_{t_{+}}(y), t_{+}}^{\ast} \vspace{1ex}\\
 & + {~} \frac{(t-t_{+})}{t} \abss{\nabla{f}(\widetilde{x}^{\ast}_{t_{+}}(y))}_{\widetilde{x}^{\ast}_{t_{+}}(y), t_{+}}^{\ast}.
\end{array}
\end{equation}
Now, since $\nabla^2{\psi_{t_{+}}}(x) \succeq \nabla^2{f}(x)$, we can show that
\begin{equation*}
\abss{\nabla{f}(\widetilde{x}^{\ast}_{t_{+}}(y))}_{\widetilde{x}^{\ast}_{t_{+}}(y), t_{+}}^{\ast} \leq \left[\nabla{f}(\widetilde{x}^{\ast}_{t_{+}}(y))^{\top}\nabla^2{f}(\widetilde{x}^{\ast}_{t_{+}}(y))^{-1}\nabla{f}(\widetilde{x}^{\ast}_{t_{+}}(y))\right]^{1/2} \leq \sqrt{\nu_f}.
\end{equation*}
Substituting this and \eqref{eq:lm3_est2} into \eqref{eq:lm3_est1}, and using the definition of $\hat{\delta}_{+}$ and $\hat{\delta}$, we get
\begin{equation*} 
\frac{\tilde{\Delta}^2}{1 + \tilde{\Delta}} \leq \tilde{\Delta}\hat{\delta} + \left(\frac{\hat{\delta}_{+}t_{+}}{t} + \frac{(t - t_{+})}{t}\sqrt{\nu_f}\right)\tilde{\Delta}_{+}.
\end{equation*}
Finally, using $t_{+} = \sigma t$, we obtain the first estimate of \eqref{eq:key_est03} from the last inequality.

Similarly, by following the same argument as in the proof of the first estimate in \eqref{eq:key_est03}, we can show that
\begin{equation*} 
\frac{\tilde{\Delta}_{+}^2}{1 + \tilde{\Delta}_{+}} \leq \tilde{\Delta}_{+}\hat{\delta}_{+} + \left(\frac{\hat{\delta}t}{t_{+}} + \frac{(t - t_{+})}{t_{+}}\sqrt{\nu_f}\right)\tilde{\Delta},
\end{equation*}
which is the second estimate of  \eqref{eq:key_est03}.

Assume that we choose $\hat{\delta} \leq \delta$ and $\hat{\delta}_{+} \leq \delta$ for some $\delta \in (0, 1)$.
Since $t_{+} = \sigma t$, if we denote by $c_{\nu}(\sigma) := \frac{\delta}{\sigma} + \frac{(1 - \sigma)}{\sigma}\sqrt{\nu_f} \in (0, 1)$.
Assume further that $\delta + c_{\nu}(\sigma) < 1$.
Then, it is clear that $\frac{\hat{\delta}_{+}t_{+}}{t} + \frac{(t - t_{+})}{t}\sqrt{\nu_f} \leq c_{\nu}(\sigma)$ and $\frac{\hat{\delta}t}{t_{+}} + \frac{(t - t_{+})}{t_{+}}\sqrt{\nu_f} \leq c_{\nu}(\sigma)$.
Applying Lemma~\ref{le:key_property3}, we can see that $(\tilde{\Delta}, \tilde{\Delta}_{+})\in\Nc(\delta, c_{\nu}(\sigma))$.
Hence, we have
\begin{equation*} 
\tilde{\Delta}\leq \frac{\delta + c_{\nu}(\sigma)}{1 - \delta -c_{\nu}(\sigma)} ~~~\text{and}~~ \tilde{\Delta}_{+} \leq \frac{\delta + c_{\nu}(\sigma)}{1 - \delta -c_{\nu}(\sigma)},
\end{equation*}
which proves \eqref{eq:key_est05}.
\Eproof

\beforesubsubsec
\subsection{\bf The proof of result in Subsection~\ref{subsec:initial_point}: Finding an initial point.}\label{apdx:th:initial_point}
\aftersubsubsec
The proof of Theorem~\ref{th:initial_point} requires the following key lemma.

\begin{lemma}\label{le:initial_point}
Let $\set{\hat{y}_j}$ be the sequence generated by Algorithm \ref{alg:A2}, where the step-size $\alpha_j$ is chosen such that $\alpha_j \in (0, 1]$ and $\frac{\alpha_j\hat{\lambda}_j}{1-\delta_j} < 1$.
Then, the following estimate holds
\begin{equation}\label{eq:lm51_1}
D_{t_0}(\hat{y}^{j+1}) \leq D_{t_0}(\hat{y}^{j}) - \alpha_j\big[\hat{\lambda}_j^2  - (\epsilon_j + \delta_j)\hat{\lambda}_j\big] + \omega_{\ast}\left(\frac{\alpha_j\hat{\lambda}_j}{1-\delta_j}\right),
\end{equation}
where $D_t$ is defined in \eqref{eq:smoothed_dual_prob} and $\omega_{\ast}(\tau) := -\tau - \ln(1-\tau)$.
The optimal step-size $\alpha_j$ that minimizes the right-hand side of \eqref{eq:lm51_1} is 
\begin{equation}\label{eq:step_size}
\alpha_j := \frac{(\hat{\lambda}_j - \epsilon_j - \delta_j)(1-\delta_j)^2}{\left[ 1+ (1-\delta_j)(\hat{\lambda}_j - \epsilon_j - \delta_j)\right]\hat{\lambda}_j} \in (0, 1). 
\end{equation}
The corresponding estimate from \eqref{eq:lm51_1} with this step-size is 
\begin{equation}\label{eq:lm51_2}
D_{t_0}(\hat{y}^{j+1})  \leq D_{t_0}(\hat{y}^{j})  -  \omega\left((\hat{\lambda}_j - \epsilon_j - \delta_j)(1-\delta_j)\right).
\end{equation}
In particular, if we set $\delta_j = \epsilon_j = 0$, then we get the original damped-step proximal-Newton step-size $\alpha_j = \frac{1}{1+\lambda_j}$ and the estimate $D_{t_0}(\hat{y}^{j+1}) \leq D_{t_0}(\hat{y}^{j}) -  \omega(\lambda_j)$ for $\omega(\tau) := \tau - \ln(1+\tau)$.
\end{lemma}

\begin{proof}
Firstly, from the self concordance of $d_{t_0}$ defined in \eqref{eq:smoothed_dual_prob} and $\hat{y}_{j+1} = (1-\alpha)\hat{y}_j + \alpha\hat{s}_j$, we can show that
\begin{equation}\label{eq:lm51_est1}
\begin{array}{ll}
	d_{t_0}(\hat{y}^{j+1}) + h_{t_0}(\hat{y}^{j+1}) & \leq d_{t_0}(\hat{y}^{j}) + \iprods{\nabla d_{t_0}(\hat{y}^j), \hat{y}^{j+1} - \hat{y}^j} + \omega_{\ast}(\norms{\hat{y}^{j+1} - \hat{y}^j}_{\hat{y}_j, t_0})\vspace{1ex}\\
	& + {~} (1-\alpha)h_{t_0}(\hat{y}^{j}) + \alpha h_{t_0}(\hat{s}^{j})\vspace{1ex}\\
	& = d_{t_0}(\hat{y}^{j}) + \alpha\iprods{\nabla d_{t_0}(\hat{y}^j), \hat{s}^{j} - \hat{y}^j} + \omega_{\ast}(\alpha\norms{\hat{s}^{j} - \hat{y}^j}_{\hat{y}_j, t_0})\vspace{1ex}\\
	& + {~} (1-\alpha)h_{t_0}(\hat{y}^{j}) + \alpha h_{t_0}(\hat{s}^{j})\vspace{1ex}\\
	& = (1-\alpha)\big(d_{t_0}(\hat{y}^{j}) + h_{t_0}(\hat{y}^{j})\big) + \omega_{\ast}(\alpha\norms{\hat{s}^{j} - \hat{y}^j}_{\hat{y}_j, t_0})\vspace{1ex}\\
	& + {~} \alpha\big(d_{t_0}(\hat{y}^{j}) + h_{t_0}(\hat{s}^{j}) + \iprods{\nabla d_{t_0}(\hat{y}^j), \hat{s}^{j} - \hat{y}^j}\big).
\end{array}
\end{equation}
Next, we will prove that 
\begin{equation}\label{eq:lm51_est2}
d_{t_0}(\hat{y}^{j}) + h_{t_0}(\hat{s}^{j}) + \iprod{\nabla d_{t_0}(\hat{y}^j), \hat{s}^{j} - \hat{y}^j} \leq d_{t_0}(\hat{y}^j) + h_{t_0}(\hat{y}^j) - \frac{1}{2}\lambda_j^2 - \frac{1}{2}\hat{\lambda}_j^2 + \frac{\epsilon_j^2}{2} + \delta_j \hat{\lambda}_j,
\end{equation}
where $\lambda_j := \tnorms{\hat{y}^{j} - s^j}_{\hat{y}_j, t_0}$. 

Indeed, by the Cauchy-Schwarz inequality, we have
\begin{equation}\label{eq:lm51_est3}
\begin{array}{ll}
\iprods{\nabla d_{t_0}(\hat{y}^j) - \widetilde{\nabla}d_{t_0}(\hat{y}^j), \hat{s}^{j} - \hat{y}^j} &\leq \tnorms{\nabla d_{t_0}(\hat{y}^j) - \widetilde{\nabla}d_{t_0}(\hat{y}^j)}_{\hat{y}^j,t_0}^{\ast} \tnorms{\hat{s}^{j} - \hat{y}^j}_{\hat{y}^j,t_0} \vspace{1ex}\\
&\leq \delta_j \hat{\lambda}_j.
\end{array}
\end{equation}
Since
\begin{equation*}
\hat{s}^j :\approx s^j := \text{prox}_{h_{t_0}}^{\widetilde{\nabla}^2d_{t_0}(\hat{y}^j)}\left( \hat{y}^j - \widetilde{\nabla}^2d_{t_0}(\hat{y}^j)^{-1}\widetilde{\nabla}d_{t_0}(\hat{y}^j) \right),
\end{equation*}
we have
\begin{equation}\label{eq:lm51_est4}
\begin{array}{ll}
\iprods{\widetilde{\nabla}d_{t_0}(\hat{y}^j), \hat{s}^j - \hat{y}^j} + h_{t_0}(\hat{s}^j) & \leq \iprods{\widetilde{\nabla}d_{t_0}(\hat{y}^j), s^j - \hat{y}^j} + h_{t_0}(s^j) \vspace{1ex}\\
& + \frac{1}{2}\tnorms{s^{j} - \hat{y}^j}_{\hat{y}^j,t_0}^2 - \frac{1}{2}\tnorms{\hat{s}^{j} - \hat{y}^j}_{\hat{y}^j,t_0}^2 + \frac{\epsilon_j^2}{2}\vspace{1ex}\\
& = \iprods{\widetilde{\nabla}d_{t_0}(\hat{y}^j), s^j - \hat{y}^j} + h_{t_0}(s^j) + \frac{1}{2}\lambda_j^2 - \frac{1}{2}\hat{\lambda}_j^2 + \frac{\epsilon_j^2}{2}.
\end{array}
\end{equation}
Using $0 \in \widetilde{\nabla}d_{t_0}(\hat{y}^j) + \widetilde{\nabla}^2d_{t_0}(s^{j} - \hat{y}^j) + \partial h_{t_0}(s^j)$, we can further estimate
\begin{equation}\label{eq:lm51_est5}
\begin{array}{l}
\iprods{\widetilde{\nabla}d_{t_0}(\hat{y}^j), s^j - \hat{y}^j} + h_{t_0}(s^j) + \frac{1}{2}\lambda_j^2 - \frac{1}{2}\hat{\lambda}_j^2 + \frac{\epsilon_j^2}{2} \vspace{1ex}\\
=  \iprods{-\widetilde{\nabla}^2d_{t_0}(s^{j} - \hat{y}^j) - \nabla h_{t_0}(s^j), s^j - \hat{y}^j} + h_{t_0}(s^j) + \frac{1}{2}\lambda_j^2 - \frac{1}{2}\hat{\lambda}_j^2 + \frac{\epsilon_j^2}{2} \vspace{1ex}\\
= \iprods{\nabla h_{t_0}(s^j), \hat{y}^j - s^j} + h_{t_0}(s^j) - \frac{1}{2}\lambda_j^2 - \frac{1}{2}\hat{\lambda}_j^2 + \frac{\epsilon_j^2}{2} \vspace{1ex}\\
\leq h_{t_0}(\hat{y}^j) - \frac{1}{2}\lambda_j^2 - \frac{1}{2}\hat{\lambda}_j^2 + \frac{\epsilon_j^2}{2},
\end{array}
\end{equation}
where $\nabla{h}_{t_0}(s^j) \in  \partial h_{t_0}(s^j)$.
Combining \eqref{eq:lm51_est4} and \eqref{eq:lm51_est5}, we get
\begin{equation}\label{eq:lm51_est6}
\iprods{\widetilde{\nabla}d_{t_0}(\hat{y}^j), \hat{s}^j - \hat{y}^j} + h_{t_0}(\hat{s}^j)  \leq h_{t_0}(\hat{y}^j) - \frac{1}{2}\lambda_j^2 - \frac{1}{2}\hat{\lambda}_j^2 + \frac{\epsilon_j^2}{2}.
\end{equation}
Now, we can prove \eqref{eq:lm51_est2} as follows:
\begin{equation*}
\begin{array}{l}
d_{t_0}(\hat{y}^{j}) + h_{t_0}(\hat{s}^{j}) + \iprods{\nabla d_{t_0}(\hat{y}^j), \hat{s}^{j} - \hat{y}^j} \vspace{1ex}\\
= d_{t_0}(\hat{y}^{j}) + h_{t_0}(\hat{s}^{j}) + \iprods{\widetilde{\nabla}d_{t_0}(\hat{y}^j), \hat{s}^{j} - \hat{y}^j} + \iprods{\nabla d_{t_0}(\hat{y}^j) - \widetilde{\nabla}d_{t_0}(\hat{y}^j), \hat{s}^{j} - \hat{y}^j}\vspace{1ex}\\
\overset{\tiny\eqref{eq:lm51_est3}}{\leq} d_{t_0}(\hat{y}^{j}) + h_{t_0}(\hat{s}^{j}) + \iprods{\widetilde{\nabla}d_{t_0}(\hat{y}^j), \hat{s}^{j} - \hat{y}^j} + \delta_j \hat{\lambda}_j\vspace{1ex}\\
\overset{\tiny\eqref{eq:lm51_est6}}{\leq} d_{t_0}(\hat{y}^j) + h_{t_0}(\hat{y}^j) - \frac{1}{2}\lambda_j^2 - \frac{1}{2}\hat{\lambda}_j^2 + \frac{\epsilon_j^2}{2} + \delta_j \hat{\lambda}_j.
\end{array}
\end{equation*}
Combining \eqref{eq:lm51_est1} and \eqref{eq:lm51_est2}, and notice that $\omega_{\ast}(\alpha\norms{\hat{s}^{j} - \hat{y}^j}_{\hat{y}_j, t_0}) \leq \omega_{\ast}(\frac{\alpha\hat{\lambda}_j}{1-\delta_j})$ we can deduce 
\begin{equation*} 
d_{t_0}(\hat{y}^{j+1}) + h_{t_0}(\hat{y}^{j+1}) \leq d_{t_0}(\hat{y}^{j}) + h_{t_0}(\hat{y}^{j}) - \alpha\left(\frac{\lambda_j^2}{2} + \frac{\hat{\lambda}_j^2}{2} - \frac{\epsilon_j^2}{2} - \delta_j\hat{\lambda}_j\right) + \omega_{\ast}\left(\frac{\alpha\hat{\lambda}_j}{1-\delta_j} \right).
\end{equation*}
Using the fact that $\lambda_j \geq \hat{\lambda}_j - \epsilon_j$ and the definition $D_{t_0}  := d_{t_0}  + h_{t_0}$, we obtain \eqref{eq:lm51_1}.

Next, if we maximize $\zeta(\alpha) := \alpha\big[\hat{\lambda}_j^2  - (\epsilon_j + \delta_j)\hat{\lambda}_j\big] - \omega_{\ast}\left(\frac{\alpha\hat{\lambda}_j}{1-\delta_j}\right)$, we have $\alpha^{\star} := \frac{(\hat{\lambda}_j - \epsilon_j - \delta_j)(1-\delta_j)^2}{\left[ 1+ (1-\delta_j)(\hat{\lambda}_j - \epsilon_j - \delta_j)\right]\hat{\lambda}_j}$ as defined by \eqref{eq:step_size}. 
Plugging $ \alpha^{\star}$ into $\zeta(\alpha)$, we get
\begin{equation*}
\zeta(\alpha^{\star}) = \frac{(\hat{\lambda}_j - \epsilon_j - \delta_j)^2(1-\delta_j)^2}{1+ (1-\delta_j)(\hat{\lambda}_j - \epsilon_j - \delta_j)} - \omega_{\ast}\left(\frac{(\hat{\lambda}_j - \epsilon_j - \delta_j)(1-\delta_j)}{1 + (\hat{\lambda}_j - \epsilon_j - \delta_j)(1-\delta_j)}\right).
\end{equation*}
Since $\frac{x^2}{1+x} - \omega_{\ast}\left(\frac{x}{1+x} \right) = \omega(x)$, we finally have $\zeta(\alpha^{\star}) = 	\omega\left((\hat{\lambda}_j - \epsilon_j - \delta_j)(1-\delta_j)\right)$, which proves \eqref{eq:lm51_2}.
If $\delta_j = \epsilon_j = 0$, then $\alpha_j$ reduces to $\frac{1}{1 + \lambda_j}$ and we obtain  $D_{t_0}(\hat{y}^{j+1}) \leq D_{t_0}(\hat{y}^{j}) -  \omega(\lambda_j)$ from \eqref{eq:lm51_2}.
\Eproof
\end{proof}


\bibliographystyle{plain}

\end{document}